\documentclass[a4paper,11pt]{amsart}

\usepackage[margin=3.0cm]{geometry}
\usepackage{amsthm,amsmath,amssymb}
\usepackage[utf8]{inputenc}
\usepackage[english]{babel}
\usepackage{graphics}
\usepackage{graphicx}
\usepackage[all,cmtip]{xy}
\usepackage{multirow}
\usepackage[linktocpage]{hyperref}
\hypersetup{colorlinks=true,linkcolor=blue,citecolor=blue,filecolor=magenta,urlcolor=blue}
\usepackage{cancel}
\usepackage{tikz}
\usepackage{pgf}
\usetikzlibrary{shapes,matrix,calc,arrows,decorations.markings,knots,patterns,intersections,cd}

\usepackage{enumitem}
\makeatletter
\def\namedlabel#1#2{\begingroup
    #2%
    \def\@currentlabel{#2}%
    \phantomsection\label{#1}\endgroup
}
\makeatother

\theoremstyle{plain}
\newtheorem{theorem}{Theorem}[section]
\newtheorem{prop}[theorem]{Proposition}

\newtheorem{lemma}[theorem]{Lemma}

\theoremstyle{definition}
\newtheorem{defi}[theorem]{Definition}
\newtheorem{ex}[theorem]{Example}
\newtheorem{step}{Step}

\theoremstyle{remark}
\newtheorem{remark}[theorem]{Remark}

\newcommand{\floor}[1]{\left \lfloor #1 \right \rfloor}
\newcommand{\ceil}[1]{\left \lceil #1 \right \rceil}
\newcommand{\fpart}[1]{\left \{ #1 \right \}}
\def\Q{\mathbb{Q}}
\def\ZZ{\mathbb{Z}}
\def\QQ{\mathbb{Q}}
\def\CC{\mathbb{C}}
\def\PP{\mathbb{P}}
\def\w{\omega}

\def\cO{\mathcal{O}}
\def\cF{\mathcal{F}}
\def\cD{\mathcal{D}}
\def\cC{\mathcal{C}}
\def\cN{\mathcal{N}}
\def\cM{\mathcal{M}}

\def\b{\ell}
\def\v{v}
\def\vv{b}
\def\Si{S}
\def\s{\sigma}
\def\qa{{\boldsymbol{{G}}}}
\DeclareMathOperator{\mult}{mult}
\DeclareMathOperator{\ttop}{top}

\DeclareMathOperator{\Weil}{Weil}
\DeclareMathOperator{\Cart}{Cart}
\DeclareMathOperator{\ord}{ord}
\DeclareMathOperator{\GL}{GL}
\DeclareMathOperator{\sing}{Sing}
\DeclareMathOperator{\reg}{Reg}
\DeclareMathOperator{\coker}{coker}
\DeclareMathOperator{\ddiv}{div}

\title[Cyclic branched coverings of surfaces with abelian quotient ...]
{Cyclic branched coverings of surfaces with abelian quotient singularities}

\author[E.~Artal]{Enrique Artal Bartolo}
\address[E.~Artal Bartolo]{Departamento de Matem\'{a}ticas, IUMA \\
Universidad de Zaragoza \\
C.~Pedro Cerbuna 12, 50009, Zaragoza, Spain}
\urladdr{http://www.unizar.es}
\email{artal@unizar.es}

\author[J.I.~Cogolludo]{Jos{\'e} Ignacio Cogolludo-Agust{\'i}n}
\address[J.I.~Cogolludo-Agustín]{Departamento de Matem\'aticas, IUMA \\ 
Universidad de Za\-ra\-go\-za \\ 
C.~Pedro Cerbuna 12 \\ 
50009 Zaragoza, Spain} 
\urladdr{http://www.unizar.es}
\email{jicogo@unizar.es} 

\author[J.~Mart\'{\i}n-Morales]{Jorge Mart\'{\i}n-Morales}
\address[J.~Martín-Morales]{Centro Universitario de la Defensa, IUMA \\
Academia General Militar \\ 
Ctra.~de Huesca s/n. \\ 
50090, Zaragoza, Spain}
\urladdr{http://cud.unizar.es/martin}
\email{jorge@unizar.es}

\thanks{Partially supported by MTM2016-76868-C2-2-P and 
Gobierno de Arag{\'o}n (Grupo de referencia ``{\'A}lgebra y Geometr{\'i}a'')
cofunded by Feder 2014-2020 ``Construyendo Europa desde Arag\'on''.
The third author is also partially supported by FQM-333 ``Junta de Andaluc{\'\i}a''.}

\tikzset{commutative diagrams/.cd,
mysymbol/.style={start anchor=center,end anchor=center,draw=none}
}
\newcommand\MySymb[2][\#]{%
  \arrow[mysymbol]{#2}[description]{#1}}

\begin{document}

\begin{abstract}
In~\cite{Esnault-Viehweg82}, Esnault-Viehweg developed the theory of cyclic branched coverings $\tilde X\to X$ of 
smooth surfaces providing a very explicit formula for the decomposition of $H^1(\tilde X,\CC)$ in terms of a resolution 
of the ramification locus. Later, in~\cite{Artal94} the first author applies this to the particular case of coverings
of $\PP^2$ reducing the problem to a combination of global and local conditions on projective curves.

In this paper we extend the above results in three directions: first, the theory is extended to surfaces with abelian 
quotient singularities, second the ramification locus can be partially resolved and need not be reduced, and finally global 
and local conditions are given to describe the irregularity of cyclic branched coverings of the weighted projective plane.

The techniques required for these results are conceptually different and provide simpler proofs for the classical results.
For instance, the local contribution comes from certain modules that have the flavor of quasi-adjunction and multiplier 
ideals on singular surfaces.

As an application, a Zariski pair of curves on a singular surface is described. In particular, we prove the existence
of two cuspidal curves of degree 12 in the weighted projective plane $\PP^2_{(1,1,3)}$ with the same singularities but 
non-homeomorphic embeddings. 
This is shown by proving that the cyclic covers of $\PP^2_{(1,1,3)}$ of order 12 ramified along the curves have different
irregularity. In the process, only a partial resolution of singularities is required.
\end{abstract}

\maketitle

\section*{Introduction}
Motivated by the Riemann’s Existence Theorem and the classification of Riemann surfaces by their projection onto the 
projective line, Zariski (\cite{Zariski-problem}) started the classification of surfaces via a projection onto the 
projective plane, the study of the fundamental group of the complement of the ramification locus (a projective curve) 
and their influence on the topology of the original surface as a branched covering of the projective plane. 
He realized that not only the type of singularities of the branched locus was relevant, but their position as well 
(\cite{Zariski-irregularity}).
In particular, he proved that the cyclic branched cover of an irreducible curve of degree $6d$ with only nodes and cusps 
is irregular, it has non-trivial first cohomology group, if the \emph{effective dimension} of the space of curves of 
degree $5d-3$ passing through the cusps is larger than its \emph{expected} (or \emph{virtual}) dimension. This difference 
was called \emph{superabundance}. More precise descriptions of the irregularity of cyclic branched coverings of curves 
in $\PP^2$ have been given by Libgober~\cite{Libgober-alexander}, Esnault~\cite{es:82}, 
Loeser-Vaqui\'e~\cite{Loeser-Vaquie-Alexander}, and Sabbah~\cite{Sabbah-Alexander}, in general smooth surfaces by 
Esnault-Viehweg~\cite{Esnault-Viehweg82} or even for general abelian branched coverings by 
Libgober~\cite{Libgober-characteristic}. 
It is worth pointing out that very concrete formulas were given in~\cite{Artal94,Libgober-characteristic} 
for the particular case of curves on $\PP^2$. These formulas combine a local 
ingredient coming from a resolution of the singularities of the branched locus and a global one measuring the 
superabundance of a certain linear system of curves on~$\PP^2$. 

In this paper we address the problem of describing the irregularity of cyclic branched coverings of surfaces with 
abelian quotient singularities and use this description to find formulas for the particular case of the weighted 
projective plane. Recall that an abelian quotient singular point in dimension two is necessarily a cyclic singularity.
The main result of this paper is presented in Theorem~\ref{thm:conucleo_singular}, where we describe the dimension 
of the equivariant spaces of the first cohomology of a $d$-cyclic cover $\rho:\tilde X\to \PP^2_w$ ramified along a
(not necessarily reduced) curve $\mathcal{C} = \sum_j n_j \mathcal{C}_j$. The cover $\rho$ naturally defines a 
divisor $H$ such that $dH$ is linearly equivalent to $\mathcal{C}$. If $K_{\mathbb{P}_w^2}$ denotes the canonical 
divisor of $\mathbb{P}_w^2$ and 
\[
\mathcal{C}^{(k)} = \sum_{j=1}^r \floor{\frac{kn_j}{d}} \mathcal{C}_j, \qquad 0 \leq k < d,
\]
then these dimensions are given as the cokernel of the evaluation linear maps
\[
\pi^{(k)}: H^0\left(\PP^2_w,\mathcal{O}_{\PP^2_w}\left( kH+K_{\PP^2_w} - \mathcal{C}^{(k)}\right) \right) 
\longrightarrow \bigoplus_{P \in S}
\frac{\mathcal{O}_{\PP^2_w,P}\left( kH+K_{\PP^2_w} - \mathcal{C}^{(k)}\right)}{\mathcal{M}_{\mathcal{C},P}^{(k)}}
\]
where $\mathcal{M}_{\mathcal{C},P}^{(k)}$ is defined as the following quasi-adjunction-type
$\mathcal{O}_{\PP^2_w,P}$-module
\[
\mathcal{M}_{\mathcal{C},P}^{(k)}\!:=\!
\left\{ g \in\mathcal{O}_{\PP^2_w,P}\left( kH+K_{\PP^2_w} - \mathcal{C}^{(k)}\right)
\vphantom{\sum_{j=1}^r}\!\right.\!
\left|\ \mult_{E_\v} \pi^* g > 
\sum_{j=1}^r \fpart{\frac{kn_j}{d}} m_{\v j} -\! \nu_\v, \ \forall \v \in \Gamma_P\! \right\}\!.
\]
The symbol $\{\cdot\}$ denotes the decimal part of a rational number and 
the multiplicities $m_{\v j}$ and $\nu_\v$ are provided by $\pi^{*} \mathcal{C}_j = \hat{\mathcal{C}}_j
+ \sum_{P \in \Si} \sum_{\v \in \Gamma_P} m_{\v j} E_\v$ and $K_{\pi} = \sum_{P \in S} \sum_{\v \in \Gamma_P}
(\nu_\v-1) E_\v$ for an embedded $\mathbb{Q}$-resolution $\pi$ of $\mathcal{C} \subset \mathbb{P}^2_w$,
cf.~Definition~\ref{def:M}.

As a consequence, 
\begin{equation}\label{eq:h1}
h^1(\tilde X, \CC)=2\sum_{k=0}^{d-1} \dim \coker \pi^{(k)}.
\end{equation}
These formulas
also reminisce the local and global interplay of conditions on linear systems on the base surface. Also, the local 
conditions can be obtained from a $\Q$-resolution of the singularities, which in particular allows for simpler 
theoretical and practical algorithms to calculate the irregularity. Moreover, in this paper the ramification along 
each irreducible component need not be the same, which translates into considering a non-reduced curve as a ramification
locus. This allows for general formulas for characteristic polynomials of the monodromy of non-reduced curves and
calculations of twisted Alexander polynomials of the complement $M$ of the curve associated with general epimorphisms 
$\pi_1(M)\to \ZZ$. 

From a purely topological point of view, it is worth highlighting the key role of singularities of the surface when 
studying coverings. This is described in Examples~\ref{ex:4A2} and \ref{ex:cusp23}. In orther words, coverings of a 
singular surface might be forced to ramify along the exceptional divisors of a resolution of the singularities of the 
surface.

As a non trivial application, in section~\ref{sec:zariski-pair} we present a Zariski pair of irreducible curves on a 
weighted projective plane, that is, two curves in the same plane $\PP^2_w$ with the same degree and local type of singularities, 
but whose embeddings are not homeomophic. In particular, we present two cuspidal curves of degree 12 in $\PP_{(1,1,3)}^2$ 
with the same local type of singularities, use the fact that~\eqref{eq:h1} is an invariant of the pair 
$(\PP^2_w,\cC)$, and calculate the different irregularity of the branched coverings of the singular surface 
ramified along the curves. The calculation of this irregularity boils down to the study of the superabundance 
of curves of degree $5$ in~$\PP_{(1,1,3)}^2$.

The following is an outline of the paper: in section~\ref{sec:settings} a general discussion on $\Q$-divisors on normal 
surfaces is given together with a review of the classical results on branched coverings of smooth surfaces. 
In section~\ref{sc:Esnault} the theory of Esnault-Viehweg is extended to branched coverings of surfaces with cyclic 
quotient singularities ramified along (non-necessarily reduced) divisors. 
This serves as a motivation to introduce certain local invariants in section~\ref{sec:local}
associated with a Weil divisor on a quotient surface singularity. Section~\ref{sec:global} is devoted 
to proving the main result discussed above in Theorem~\ref{thm:conucleo_singular} for the 
irregularity of a branched covering of the weighted projective plane.
Finally, in section~\ref{sec:examples} examples and applications are given. 
For instance local-global techniques are used to provide a proof that the invariant defined in section~\ref{sec:local}
is independent of the $\Q$-resolution. Also, a Zariski pair of curves in the projective plane is presented. 
The proof of the existence of such a pair is based on the fact that certain branched covers ramified along the curves
have different irregularity.

\subsection*{Acknowledgments}
The second and third authors want to thank the Fulbright 
Program\footnote{Salvador de Madariaga PRX18/00469 (second author) and Jos\'e Castillejo CAS18/00473 (third author) 
grants by Ministerio de Educaci\'on, Cultura y Deporte.}
for their financial support while writing this paper. They also want to thank the University of Illinois at Chicago, 
especially Anatoly Libgober, Lawrence Ein, and Kevin Tucker for their warm welcome and their support in hosting them.

We also want to thank the anonymous referee for their valuable suggestions that helped improve this paper.

\numberwithin{equation}{section}

\section{Settings}
\label{sec:settings}
In this section we will fix some notation and recall basic definitions and results from the theory of 
normal surfaces, $\QQ$-divisors, and intersection theory with special attention paid to quotient singularities,
in particular cyclic quotients. Most of these results can be found in the literature
but they will be briefly discussed here not only for the sake of completeness, but also in order to
clarify concepts and notation that might be ambiguous depending on the source. An extended version 
of the first part on quotient singularities can be found in~\cite{AMO-Intersection}, whereas the 
concepts and notation for line bundles and sheaf cohomology of rational divisors can be mainly found 
in~\cite{Mumford-topology,Sakai84}.

\subsection{Cyclic quotient surface singularities}
\label{subsec:CyclicQuotient}
\mbox{}

A quotient surface singularity for us will be a local chart that is analytically isomorphic to
$(\CC^2/G,0)$ for a finite linear group $G\subset \GL(2,\CC)$. If $G$ is a cyclic group, then 
we refer to this as a cyclic quotient singularity. A complex analytic surface that admits an open 
covering $\{U_\lambda\}$ whose local charts are (cyclic) quotient surface singularities is called a 
\emph{(cyclic) $V$-surface}. Projective algebraic $V$-surfaces are in particular normal surfaces.

A standard notation for a cyclic quotient surface singularity is $X=\frac{1}{d}(p,q)$ for pairwise coprime 
positive integers $d,p,q>0$, meaning $\CC^2/\qa_{d,p,q}$ where 
$\qa_{d,p,q}=\left\langle \left( 
\begin{smallmatrix}
\zeta_d^{p}&0\\0&\zeta_d^{q} 
\end{smallmatrix}
\right)\right\rangle$ 
and $\zeta_d\in \CC$ is a primitive $d$-th root of unity. 

In this context, the local ring $\cO_X$ of functions on $X$ is the ring of invariants of $\qa_{d,p,q}$, namely,
$$\cO_X=\CC\{x,y\}^{\qa_{d,p,q}}=\{h\in \CC\{x,y\} \mid h(g\cdot (x,y))=h(x,y) \text{ for all } g\in \qa_{d,p,q}\}.$$

Note that $d$, the order of $\qa_{d,p,q}$, does not completely determine the type of the cyclic singularity.
For instance $\cO_{\frac{1}{3}(1,2)}=\langle 1,x^3,y^3,xy\rangle\cO_X\cong \CC\{u,v,w\}/(uv=w^3)$ defines
locally a complete intersection singularity, whereas 
$\cO_{\frac{1}{3}(1,1)}=\langle 1,x^3,y^3,xy^2,x^2y\rangle\cO_X$ does not.

Morphisms between $V$-surface singularities are defined as morphisms $\CC^2\to \CC^2$ that respect the group action.

Other interesting objects associated with a cyclic quotient singularity are the modules of quasi-invariant germs.
These are defined as the eigenspaces of the operator defined by the action of $\qa_{d,p,q}$ on the coordinates, that is,
$$
\cO_{X,\zeta_d^k}:=\{h\in \CC\{x,y\} \mid h(g\cdot (x,y))=\zeta_d^k h(x,y) \text{ for all } g\in \qa_{d,p,q}\}.
$$
Note that $\cO_{X,\zeta_d^k}$ has a natural structure of $\cO_X$-submodule of $\CC\{x,y\}$.
Also $\cO_X=\cO_{X,\zeta_d^0}$ and $\CC\{x,y\}=\bigoplus_{k=0}^{d-1}\cO_{X,\zeta_d^k}$. 
This latter decomposition is associated with the $d$-th cyclic cover $(\CC^2,0)\to (X,0)$.
Moreover, this shows that any effective Weil divisor on $X$ is given by the set of zeros of a 
quasi-invariant germ, in particular an effective Cartier divisor is defined by a germ in $\cO_X$. Also, 
any two effective 
Weil divisors $D_1=(f_1)$, $D_2=(f_2)$, with $f_1,f_2\in \cO_{X,\zeta_d^k}$, are linearly equivalent, since
$\frac{f_1}{f_2}$ is a meromorphic function. Since in the local context all Cartier divisors are principal,
the divisor class group of $X$ is $\text{Weil}(X)/\text{Cart}(X)$ which in the cyclic case is isomorphic 
to~$\qa_{d,p,q}\cong\ZZ_d$.

Once a primitive root of unity $\zeta_d$ has been fixed, for a Weil divisor $D$ defined as the set of zeros
of a quasi-germ~$f$, we say that the \emph{class} $[D]=[f]$ of $D$ (or $f)$ is $k\in\mathbb{Z}_d$ 
if $f\in\cO_{X,\zeta_d^k}$.

In this context, we will refer to an irreducible Weil divisor $D=(f)$ as \emph{quasi-smooth} if $f$ defines 
a smooth germ in $\CC\{x,y\}$.

Throughout this paper we will prove results about general normal surfaces, projective $V$-surfaces and 
projective cyclic $V$-surfaces. 

\subsection{Weighted projective plane}
\mbox{}

Since it will be extensively used throughout this paper, we will briefly discuss the weighted 
projective plane as an example of projective cyclic $V$-manifold.
Consider $w:=(w_0,w_1,w_2)$ a weight vector, that is, a finite set of pairwise coprime positive
integers. There is a natural action of the multiplicative group $\CC^{*}$ on
$\CC^{3}\setminus\{0\}$ given by
$$
  (x_0,x_1,x_2) \longmapsto (t^{w_0} x_0,t^{w_1} x_1,t^{w_2} x_2).
$$
The set of orbits $\frac{\CC^{3}\setminus\{0\}}{\CC^{*}}$ under this action is denoted by $\PP^2_w$ 
and it is called the {\em weighted projective plane} of type $w$. 
Equivalently, $\PP^2_w$ can be defined starting with $\PP^2$, the classical projective plane, and 
$\qa_w = \qa_{w_0}\times \qa_{w_1}\times \qa_{w_2}$ the product of cyclic groups. Consider the group action
$$
\begin{array}{rcl}
\qa_w \times \PP^2 \quad & \longrightarrow & \quad \PP^2, \\[0.15cm]
\big( (\zeta_{w_0},\zeta_{w_2},\zeta_{w_2}), [X_0:X_1:X_2] \big) & \mapsto &
[\zeta_{w_0}X_0:\zeta_{w_1}X_1:\zeta_{w_2}X_2].
\end{array}
$$
The set of all orbits $\PP^2/\qa_w$ is isomorphic to $\PP^2_w$ and 
the isomorphism is induced by the branched covering
\begin{equation}\label{covering}
   \PP^2\ni [X_0:X_1:X_2] \overset{\phi}{\longmapsto}  [X_0^{w_0}:X_1^{w_1}:X_2^{w_2}]_w \in\PP^2_w.
\end{equation}

Note that this branched covering is unramified over
$$
\PP^2_{w} \setminus \{ [X_0,X_1,X_2]_{w} \mid X_0\cdot X_1\cdot X_2 = 0 \}
$$
and has $\bar{w}=w_0w_1w_2$ sheets. 
Moreover, one can decompose $\PP^2_w = U_0 \cup U_1 \cup U_2$,
where $U_i$ is the open set consisting of all elements $[X_0:X_1:X_2]_w$ with $X_i\neq 0$. The map
$$
  \widetilde{\psi}_2: \CC^2 \longrightarrow U_2,\quad
\widetilde{\psi}_2(x_0,x_1):= [x_0:x_1:1]_w
$$
induces an isomorphism $\psi_2$ replacing $\CC^2$ by $\frac{1}{w_2}(\w_0,\w_1)$.
Analogously, $\frac{1}{w_0}(\w_1,\w_2) \cong U_0$ and $\frac{1}{w_1}(w_2,\w_0) \cong U_1$
under the obvious analytic maps.
This shows that $\PP^2_w$ is a cyclic $V$-manifold. 

The following result justifies why the weights can be assumed to be pairwise coprime.

\begin{prop}[cf.~\cite{Dolgachev82}]
\label{propPw}
Consider $w:=(w_0,w_1,w_2)\in \mathbb{Z}^3_{> 0}$ such that $\gcd(w)=1$ and define $d_0 := \gcd (w_1,w_2)$, 
$d_1 := \gcd (w_0,w_2)$ , $d_2 := \gcd (w_0,w_1)$,
$e_0:= d_1\cdot d_2$,  $e_1:= d_0\cdot d_1$, $e_2:= d_0\cdot d_1$ 
and $\eta_i:=\frac{w_i}{e_i}$, $\eta:=(\eta_0,\eta_1,\eta_2)$. The following map is an isomorphism:
$$
\begin{array}{rcl}
\PP^2_{w} & \longrightarrow & \PP^2_{\eta}, \\[0.15cm]
\,[X_0:X_1:X_2] & \mapsto &
\big[\,X_0^{d_0}:X_1^{d_1}:X_2^{d_2}\,\big].
\end{array}
$$
\end{prop}

\begin{remark}
Note that, due to the preceding proposition, one can always assume that the 
vector $(w_0,w_1,w_2)$ is a weight vector, that is, its coordinates are pairwise coprime. 
Note that following a similar argument, $\PP^1_{(w_0,w_1)} \cong \PP^1$.
\end{remark}

\subsection{Embedded \texorpdfstring{$\Q$}{Q}-resolutions}\label{sec:Qres}
\mbox{}

Classically an embedded resolution of $\{f=0\} \subset (\CC^n,0)$ is a proper map $\pi: X \to (\CC^n,0)$ 
from a smooth variety $X$ satisfying, among other conditions, that $\pi^{-1}(\{f=0\})$ is a normal 
crossing divisor. For a weaker notion of resolution one can relax the condition on the preimage of the singularity 
by allowing the ambient space $X$ to be a projective $V$-manifold and the divisor $\pi^{-1}(\{f=0\})$ to 
have ``normal crossings'' in $X$. This notion of $\Q$-normal crossing divisor on 
$V$-manifolds was first introduced by Steenbrink in~\cite{Steenbrink77}.

\begin{defi}
Let $X$ be a cyclic $V$-surface, then a Weil divisor $D$ on $X$ is said to be with {\em $\Q$-normal crossings} 
if for any $P \in X$, there is a local isomorphism $(X,P) \simeq \frac{1}{d}(p,q)$ such that the image of
$(D,P) \subset (X,P)$ is given by local equations $\{x^{i} y^{j} = 0 \}\subset \frac{1}{d}(p,q)$, $i,j\geq 0$.
If in addition, the image of each coordinate corresponds to a different global irreducible component we say
$D$ is a {\em simple $\Q$-normal crossing} divisor.
\end{defi}

\begin{ex}
\label{ex:QNC}
As mentioned above for $X_1:=\frac{1}{3}(1,1)$ one can check that $\cO_{X_1}$ is the $\CC$-algebra generated
by $\{1,x^3,y^3,x^2y,xy^2\}$, whereas $\cO_{X_1,\zeta^1}=\langle x,y\rangle\cO_{X_1}$ and
$\cO_{X_1,\zeta^2}=\langle x^2,xy,y^2\rangle\cO_{X_1}$. The germ $\{xy=0\}$ defines a $\Q$-normal crossing
divisor $D=(xy)$ whose class $[D]$ in $\Weil(X_1)/\Cart(X)$ is 2 and is a union of quasi-smooth divisors 
$D=D_1+D_2$ with~$[D_1]=[D_2]=1$.

Analogously, for $X_2:=\frac{1}{3}(1,2)$ one has $\cO_{X_2}=\langle 1,x^3,y^3,xy\rangle$, 
$\cO_{X_2,\zeta^1}=\langle x,y^2\rangle\cO_{X_2}$,
$\cO_{X_2,\zeta^2}=\langle x^2,y\rangle\cO_{X_2}$. 
The germ $\{xy=0\}$ defines a $\Q$-normal crossing divisor $D=(xy)\in \Cart(X)$ and it is a union of quasi-smooth 
divisors $D=D_1+D_2$ with~$[D_1]=1$, $[D_2]=2$.
\end{ex}

\begin{defi}\label{Qresolution}
Let $X$ be a cyclic $V$-surface. Consider $D \subset X$ an analytic 
subvariety of codimension one. An embedded $\Q$-resolution of $D \subset X$ 
is a proper analytic map~$\pi: Y \to X$ such that:
\begin{enumerate}
\item $Y$ is a cyclic $V$-manifold,
\item $\pi$ restricted to $Y\setminus \pi^{-1}(\sing(D))$ is an isomorphism, and
\item $\pi^{-1}(D)$ is a $\Q$-normal crossing divisor on $Y$.
\end{enumerate}
\end{defi}

Usually one uses weighted or toric blow-ups with smooth center as a tool for finding embedded 
$\Q$-resolutions. Here we will briefly discuss weighted blow-ups in the surface case   
$\pi: \widehat{X} \to X$ at a point $P\in X=\frac{1}{d}(p,q)$ with respect to $w = (a,b)$. 
Consider
\[
\hat X:=\{((x,y),[u:v]_w)\in \CC^2\times\PP^1_w\mid x=t^a u, y=t^b v \text{ for some } t\in \CC\}/\qa_{d,p,q}
\]
and $\pi: \widehat{X} \to X$ the projection map on the first coordinate, as usual. For practical 
reasons, we will give a specific description of $\hat X$ as a cyclic $V$-surface. 
The new space $\hat X$ is covered by two open sets $\widehat{U}_1 \cup \widehat{U}_2$ each one of them 
is a quotient singularity of type:
\begin{equation}
\label{eq:charts} 
\array{ccc}
\widehat{U}_1 \cong
X \left( \displaystyle\frac{ad}{e}; 1, \frac{-b+p' aq}{e} \right)
& \text{ and } &
\widehat{U}_2 \cong 
X\left( \displaystyle\frac{bd}{e}; \frac{-a+q' bp}{e}, 1\right),
\endarray
\end{equation}
where $p'p=q'q\equiv 1 \mod (d)$ and $e := \gcd(d,aq-bp)$. Finally, the charts are given by

\begin{center}
$\begin{array}{c|c}
X \left( \displaystyle\frac{ad}{e}; 1, \frac{-b+p' aq}{e} \right)  \ \longrightarrow \ 
\widehat{U}_1, &X \left( \displaystyle\frac{bd}{e}; \frac{-a+q' bp}{e}, 1
\right) \ \longrightarrow \ \widehat{U}_2, \\[0.2cm] \,\big[ (x^e,y) \big] \mapsto 
\big[ ((x^a,x^b y),[1:y]_{w}) \big]_{(d;p,q)} & \big[ (x,y^e) \big] \mapsto 
\big[ ((x y^a, y^b),[x:1]_{w}) \big]_{(d;p,q)}.
\end{array}$
\end{center}

\subsection{Intersection theory on projective \texorpdfstring{$V$}{V}-surfaces}
\mbox{}

In this section we will briefly recall the definitions of intersection theory on projective $V$-surfaces.
We denote by $\Weil_{\QQ}(X)$ the vector space $\Weil(X)\otimes \QQ$ and refer to their elements as $\QQ$-divisors. 
As a general reference for intersection theory in normal varieties one can use~\cite{Sakai84,Fulton-Intersection}, however 
in the particular setting of $V$-surfaces we will follow the notation given in~\cite{AMO-Intersection}.
For the sake of completeness, we will briefly present the classical definitions of intersection multiplicity
of two divisors on a $V$-surface assuming the theory on smooth surfaces.

\begin{defi}
\label{def:multintVM}
Let $X$ be a $V$-surface and consider $D_1, D_2 \in \Cart(X)$ Cartier divisors.
The \emph{intersection number} $(D_1 \cdot D_2)_X$ is defined as 
\begin{equation}
\label{eq:MI} 
(D_1 \cdot D_2)_X:=(\pi^*(D_1)\cdot \pi^*(D_2))_Y,
\end{equation}
where $\pi:Y\to X$ is a resolution of $X$, whenever the right-hand side of~\eqref{eq:MI} is well defined.

If $D_1, D_2 \in \Weil(X)$ are irreducible divisors such that $k_1D_1, k_2D_2\in \Cart(X)$ and 
either $D_1$ is compact or $D_1\cap D_2$ is finite, then  
$$(D_1 \cdot D_2)_X :=\frac{1}{k_1 k_2} (k_1 D_1 \cdot k_2 D_2 )_X.$$
By bilinearity and tensoring by $\QQ$ these definitions are extended to any two 
rational divisors on~$X$.
\end{defi}

Note that if $D_1, D_2 \in \Weil(X)$ are effective and $D_1\cap D_2$ is finite, 
then the following formula holds
$$
(D_1 \cdot D_2)_X =\sum_{P\in X}\frac{1}{k_1 k_2} (k_1 D_1 \cdot k_2 D_2 )_P,
$$
where $(k_1D_1 \cdot k_2D_2)_P:=\dim_\CC \frac{\cO_{X,P}}{(f_1,f_2)}$ and $k_1 D_1$
(resp. $k_2 D_2$) is the support of a germ~$f_1\in \cO_{X,P}$ (resp. $f_2\in \cO_{X,P}$).

\begin{theorem}
\label{thm:intersection}
Let $F : Y \to X$ be a proper dominant morphism between two irreducible $V$-surfaces, and
$D_1, D_2 \in \Weil_\QQ(X)$.
\begin{enumerate}[label=\rm(\arabic*)]
 \item\label{item:1}
 The cardinal of $F^{-1}(P)$, $P \in X$ generic, is a finite constant. This finite number is denoted by $\deg(F)$.
 \item 
 If $(D_1 \cdot D_2)_X$ is defined, then so is the number $(F^*(D_1) \cdot F^*(D_2))_Y$. In that case
 $(F^*(D_1) \cdot F^*(D_2))_Y=\deg(F)(D_1 \cdot D_2)_X$.
 \item 
 If $(D_1 \cdot D_2)_P$ is defined for some generic $P \in X$ as in~\ref{item:1}, then so is $(F^*(D_1) \cdot F^*(D_2))_Q$, 
 $\forall Q \in F^{-1}(P)$, and 
 $\sum_{Q\in F^{-1}(P)}(F^*(D_1)\cdot F^*(D_2))_Q=\deg(F)(D_1\cdot D_2)_P$.
\end{enumerate}
\end{theorem}

\begin{ex}
\label{ex:QNC:intersection}
Following Example~\ref{ex:QNC}, note that $3D_1, 3D_2\in \Cart(X_1)$ and hence, according to Definition~\ref{def:multintVM},
one has $(3D_1\cdot 3D_2)_{X_1}=\dim_\CC\frac{\cO_{X_1}}{(x^3,y^3)}=\dim_\CC \langle 1,x^2y,xy^2\rangle_\CC=3$. 
Hence $(D_1\cdot D_2)_{X_1}=\frac{1}{3}$.

Analogously one can check $(3D_1\cdot 3D_2)_{X_2}=\dim_\CC \langle 1,xy,x^2y^2\rangle_\CC=3$, and thus also
$(D_1\cdot D_2)_{X_2}=\frac{1}{3}$.
\end{ex}

\begin{prop}\label{formula_self-intersection}
Let $X$ be a cyclic $V$-surface. Let $\pi: \hat{X} \to X$ be the weighted blow-up at a point $P$ 
of type $\frac{1}{d}(p,q)$ with respect to $w=(a,b)$. Assume $(d,p)=(d,q)=(a,b)=1$ and write $e = \gcd(d,aq-bp)$.

Consider $C$ and $D$ two $\QQ$-divisors on $X$, denote by $E$ the exceptional divisor of $\pi$, 
and by $\widehat{C}$ (resp.~$\widehat{D}$) the strict transform of $C$ (resp.~$D$). Let $\nu_{C,P}$ 
(resp.~$\nu_{D,P}$) be the $w$-multiplicity of $C$ (resp $D$) at $P$, (defined such that 
$\nu_{x,P}=a$ and $\nu_{y,P}=b$). Then the following equalities hold:
\begin{enumerate}[label=\rm(\arabic*)]
 \item\label{formula_self-intersection1} 
 $\displaystyle \pi^{*}(C) = \widehat{C} + \frac{\nu_{C,P}}{e} E$,
 \item\label{formula_self-intersection2} 
 $\displaystyle (E \cdot \widehat{C})_{\hat X} = \frac{e \nu_{C,P}}{a b d}$,
 \item\label{formula_self-intersection3} 
 $\displaystyle (C \cdot D)_P=\widehat{C} \cdot \widehat{D}+ \frac{\nu_{C,P}\nu_{D,P}}{a b d}$,
 \item\label{formula_self-intersection4} 
 $\displaystyle (E \cdot E)_{\hat X}=-\frac{e^2}{dab}$.
\end{enumerate}
\end{prop}

\begin{ex}
\label{ex:QNC:resolution}
Consider $\pi_1$ the blow-up with $w=(1,1)$ of $X_1$ from Example~\ref{ex:QNC}, 
using~Proposition~\ref{formula_self-intersection} one obtains
$\pi_1^*D_1=\hat D_1 + \frac{1}{3} E$, $(\hat D_i\cdot E)_{\hat X_1}=1$, $(\hat D_1\cdot \hat D_2)_{\hat X_1}=0$,
and $(E\cdot E)_{\tilde X_1}=-3$. Moreover, by~\eqref{eq:charts} $\tilde X_1$ is a smooth surface.
\end{ex}

\subsection{Rational divisors on normal surfaces}
\mbox{}

Throughout this section we will assume $X$ is a reduced, irreducible, complex projective variety of 
dimension 2. Given any $D=\sum \alpha_iC_i \in \Weil_{\QQ}(X)$, where $C_i$ are pairwise distinct irreducible
curves and $\alpha_i\in \QQ$ we use the following notation:
$$
\array{l}
\left\lfloor D\right\rfloor=\sum \lfloor \alpha_i\rfloor C_i,\\
\left\lceil D\right\rceil=\sum \lceil \alpha_i\rceil C_i.\\
\endarray
$$
If $\pi:Y\to X$ is a $\Q$-resolution of $X$, then the total transform of $D$ by $\pi$ can be written as 
$$\pi^*D=\hat D + \sum_{\v\in \Gamma} m_\v E_\v,$$
where $\Gamma$ is the dual graph of the $\Q$-resolution, $\hat D$ the strict transform of~$D$, and 
$m_\v\in \QQ$ are uniquely determined by a solution to the linear system 
\begin{equation}
\label{eq:pullback}
(E_\v\cdot \pi^*D)_Y=0 \text{ for all } \v\in \Gamma.
\end{equation}
Note that~\eqref{eq:pullback} matches Theorem~\ref{thm:intersection}.

\begin{ex}
\label{ex:QNC:pullback}
Following Example~\ref{ex:QNC:resolution} note that $\pi_1^*D_1=\hat D_1+mE$ and the equation 
$E\cdot \pi_1^*D_1=E\cdot \hat D_1+mE^2=1-3m=0$ matches the solution $m=\frac{1}{3}$ given by Theorem~\ref{thm:intersection}.
\end{ex}

Given $D\in \Weil(X)$, the sheaf $\cO_X(D)$ is defined as $i_*\cO_{\reg(X)}(D|_{\reg(X)})$ via the 
inclusion $i:\reg(X)\to X$ of the regular part of $X$. In the general case $D\in \Weil_\QQ(X)$ one defines
\begin{equation}
\label{eq:OD}
\cO_X(D):=\cO_X(\lfloor D \rfloor).
\end{equation}
As a word of caution, note that this notation differs from the one used in~\cite{Nemethi-Poincare}.
Using this definition the following projection formula is obtained.

\begin{theorem}[{Projection formula~\cite[Thm. 2.1]{Sakai84}}]
\label{thm:projection}
If $D\in \Weil_\QQ(X)$ and $\pi:Y\to X$ is a resolution of $X$, then
$$
\pi_*(\cO_Y(\pi^*D))=\cO_X(D).
$$
\end{theorem}
This shows how $\cO_X(D)$ can be interpreted via the divisor $D$.

\begin{prop}\label{prop:H0D}
Let $X$ be a normal surface and $D \in \Weil_{\QQ}(X)$. Then,
the cohomology $H^ 0(X, \cO_X(D))$ can be identified with
$\{ h \in K(X) \mid (h) + D \geq 0 \}$.
\end{prop}

\begin{proof}
The statement of this result is true for an integral divisor on a smooth variety.
Let $\pi: \tilde{X} \to X$ be a resolution of $X$ consisting in a composition of blowing-ups.
Then, by~\cite[Theorem 1.2]{Sakai84}, 
$\cO_X(D) = \pi_{*} \cO_{\tilde{X}}(\pi^{*} D) = \pi_{*} \cO_{\tilde{X}}(\floor{\pi^{*} D})$
and thus $H^0 (X,\cO_X(D)) = H^0(\tilde{X},\cO_{\tilde{X}}(\floor{\pi^{*} D}))$. Since the latter
is an integral divisor in a smooth variety,
$$
H^0(\tilde{X},\cO_{\tilde{X}}(\floor{\pi^{*} D})) = \{ \tilde h \in K(\tilde{X}) \mid (\tilde h) + \floor{\pi^* D} \geq 0 \}.
$$
The condition $(\tilde h) + \floor{\pi^* D} \geq 0$ is equivalent to $(\tilde h) + \pi^{*} D \geq 0$ 
because $(\tilde h)$ is an integral divisor. Writing $\tilde h = \pi^{*} h$ where $h \in K(X)$, 
one can split the previous condition $(\pi^{*} h) + \pi^{*} D \geq 0$ into two different ones,
\begin{equation}\label{eq:global_sections}
(h) + D \geq 0 \quad \text{ and} \quad \ord_E \left( (\pi^{*} h) + \pi^{*} D \right) \geq 0,
\end{equation}
for all $E$ exceptional divisors of $\pi$.

Note that if $A = (a_{ij})_{i,j}$ is a negative-definite real matrix such that $a_{ij} \geq 0$, $i\neq j$,
then all entries of $-A^{-1}$ are non-negative. This implies that the pull-back of an effective divisor
is also an effective divisor. Hence one can easily observe that the second condition in~\eqref{eq:global_sections}
follows from the first one and the proof is complete.
\end{proof}

Finally, we will define the concept of linear equivalence for $\QQ$-divisors. Given $D_1, D_2\in \Weil_\QQ(X)$
we say $D_1$ is {\em linearly equivalent} to $D_2$ and denote it by $D_1\sim D_2$ if $D_1-D_2$ is a principal 
divisor defined by a meromorphic function on $X$. Note that:
\begin{equation}
D_1\sim D_2 \Leftrightarrow 
\begin{cases}
D_1-D_2 \text{ is an integer divisor and}\\
\cO_X(D_1)\cong \cO_X(D_2).
\end{cases}
\end{equation}

\begin{ex}
\label{ex:QNC:LEquiv}
Following Example~\ref{ex:QNC:pullback} consider $D_{11}=\frac{1}{3}D_1$ and $D_{12}=\frac{2}{3}D_1$. 
Note that $D_1\sim D_2$ since $\frac{x}{y}$ is a meromorphic function on $X_1$. However
$D_{11}\not\sim D_{12}$ despite $\cO_{X_1}(D_{11})=\cO_{X_1}(D_{12})=\cO_{X_1}.$ 
The reason here is that $D_{11}-D_{12}=\frac{1}{3}D_1-\frac{1}{3}D_2$ is not a Weil divisor.

Also, note that $\hat D_1 + \lambda E\sim \hat D_2+\lambda E$ for any $\lambda\in \QQ$ since the 
difference is a Weil divisor and $\hat D_1- \hat D_2$ is defined by composing $\frac{x}{y}$ with the 
blow-up~$\pi_1$.  

Finally, in $X_2$ the two quasi-smooth germs $D_1$ and $D_2$ are not linearly equivalent since $D_1\sim D_2$ 
implies that $[D_1]=[D_2]$ however $[D_1]=1$ and $[D_2]=2$.
\end{ex}

\subsection{Riemann-Roch formula for normal surfaces}\label{sec:RR} 
\mbox{}

For a given $X$ a normal projective surface and $D\in \Weil(X)$, the following formula is a generalization 
of the classical Riemann-Roch formula from the smooth case.

\begin{theorem}[\cite{Blache95}]
\label{thm:RR}
There is a rational map $R_{X,P}:\Weil(X,P)/\Cart(X,P)\to \QQ$ for each singular point $P\in \sing(X)$ such that 
$$\chi(\cO_X(D))-\chi(\cO_X)=\frac{1}{2}D\cdot (D-K_X) + \sum_{P\in \sing(X)} R_{X,P}(D).$$
\end{theorem}
A consequence of this result via Serre duality is that $R_{X,P}(D)=R_{X,P}(D-K_X)$.
One major breakthrough accomplished by Blache in~\cite[Thm.~2.1]{Blache95} is to provide an interpretation for
$R_{X,P}(D+K_{X})$ as follows. Consider $\sigma:(\tilde X,E)\to (X,P)$ a resolution of $X$ at $P$, where $E$ is the 
exceptional part of the resolution, $\sigma^* D=E_{D}+\hat D$, $\hat D$ is the strict transform of $D$, $E_{D}$ 
is its exceptional part, and $K_{\tilde X}$ (resp.~$K_X$) is the canonical divisor of $\tilde{X}$ (resp.~$X$). Then
\begin{equation}
\label{eq:A_X}
A_{X,P}(D):=-R_{X,P}(D+K_{X})=\frac{1}{2}E_{D}(\hat D+K_{\tilde X})-h^0(\sigma_* \cO_{\hat D}/{\cO_{D}}),
\end{equation}
where both summands depend on $\sigma$. Moreover, if $\sigma$ is a good resolution of $(X,D)$, then
\begin{equation}
\label{eq:delta-top}
\delta^{\ttop}(D):=\frac{1}{2}E_{D}(\hat D+K_{\tilde X})=-\frac{1}{2}E_{D}(E_D+Z_{\tilde X}),
\end{equation}
($Z_{\tilde X}$ is a numerical canonical cycle) and $\delta(D):=h^0(\sigma_* \cO_{\hat D}/{\cO_{D}})$, the classical
$\delta$-invariant, 
do not depend on~$\sigma$. Denote by $K_{\sigma} := K_{\tilde{X}} - \sigma^{*} K_X$ the relative canonical divisor,
then the canonical cycle $Z_{\tilde X}$ is numerically equivalent to $-K_{\sigma}$.

A second interpretation of this invariant is provided in~\cite{jiJM-correction}, where the second and third authors 
define a new invariant $\kappa_P(D)$, which in the cyclic $V$-manifold case has a geometric interpretation as $r_P(k)-1$,
where $r_P(k)$ is defined as the number of irreducible components of a generic Weil divisor on $(X,P)$ of class $k$, 
that is, a generic germ in~$\cO_{X,\zeta^k}$. In~\cite{jiJM-correction} it is shown that
$$
A_{X,P}(D)=\delta^{\ttop}(D)-\kappa_P(D).
$$
When applied this formula for generic germs one obtains a combinatorial way to calculate $A_{X,P}(D)$, that is,
$$
A_{X,P}(D)=\delta^{\ttop}(D)-r_P(D)+1.
$$

\subsection{Cyclic coverings of smooth algebraic surfaces d'apr\`es Esnault-Viehweg}
\mbox{}

Esnault-Viehweg's theory for cyclic covers in the smooth case can be 
presented as follows. Consider $X$ a projective smooth surface and 
let $D$ be a divisor which is linearly equivalent to $nH$ where
$H$ is another divisor. Then $\mathcal{O}_X(D)$ is isomorphic to 
$\mathcal{O}_X(H)^{\otimes n}$. Then, given a meromorphic section 
$t:X\dashrightarrow\mathcal{O}_X(D)$, such that $\ddiv (t)=D$,
we can consider
\[
\hat{X}:=\overline{\{(x,v)\in\mathcal{O}_X(H)\mid v^{\otimes n}=t(x)\}}
\]
and a suitable smooth model~$\pi:\tilde{X}\to \hat X$. Recall that the 
first Betti number $\dim_\CC H^1(\tilde{X},\CC)$ of $\tilde{X}$ equals 
$2\dim H^1(\tilde{X},\mathcal{O}_{\tilde{X}})$ 
and it is a birational invariant.

Let us assume now that $D$ is a simple normal crossing divisor. We can assume 
that $D$ is effective and $D=\sum_{j=1}^r n_i D_i$ is its decomposition
in irreducible components and $0\leq n_i<n$ (note that $n_i=0$ means that the 
covering is not ramified along $D_i$, but we allow this for technical reasons).
Using the eigen-decomposition of $\pi_*(\mathcal{O}_{\tilde{X}})$ induced
by that of $\mathcal{O}_{\tilde{X}}$ one can describe the irregularity of
$\tilde X$ in cohomological terms using line bundles on~$X$.

\begin{theorem}[\cite{Esnault-Viehweg82}]
Under the previous conditions, the irregularity of the covering $\tilde{X}$ equals 
$\sum_{k=0}^{n-1}\dim H^1(X,\mathcal{L}^{(k)})$, where
\begin{equation}
\label{eq:Lkliso} 
\mathcal{L}^{(k)}=
\mathcal{O}_X
\left(
-k H+
\sum_{i=0}^r
\left\lfloor\frac{k n_i}{n}\right\rfloor 
D_i
\right).
\end{equation}
\end{theorem}
One of the goals of this paper is to generalize this result to normal surfaces, see Theorem~\ref{thm:Esnault}.

\subsection{Cyclic coverings of \texorpdfstring{$\PP^2$}{P2} branched along a curve}
\mbox{}

In the particular case when $X=\PP^2$ the complex projective plane and $\mathcal{C}$ is a reduced projective plane 
curve of degree~$n$, the description of the $n$-th cyclic cover of $\PP^2$ ramified 
along~$\mathcal{C}$ can be given in very specific terms based on the dimension of the space of curves of 
a certain degree with a prescribed behavior at the singular points of~$\cC$.
Such a relation appears in several works and has its origin in Zariski 
(cf.~\cite{Zariski-irregularity,Libgober-alexander,es:82,Loeser-Vaquie-Alexander,Sabbah-Alexander,Artal94,Libgober-characteristic}). 

Before we can give a more precise
statement we need to introduce some notation. 
Let $\sigma:Y\to\PP^2$ be an embedded resolution of $\cC$
(in fact, we can take the minimal one and we may resolve only the 
set $\sing^*(\mathcal{C})$ of singular points of $\mathcal{C}$ other than ordinary double points 
whose branches belong to different irreducible components of $\mathcal{C}$).
As a consequence $\sigma^*(\mathcal{C})$ 
is a \emph{simple normal crossing divisor}. Let $\mathcal{C}_1,\dots,\mathcal{C}_r$ be the irreducible
components of $\mathcal{C}$. Then, one can write
\begin{equation}
\label{eq:ttC}
\sigma^*(\mathcal{C})=
\sum_{i=1}^r \hat{\mathcal{C}}_i+
\sum_{P\in\sing(\mathcal{C})}
\sum_{i=1}^{r_P} n_{P,i} E_{P,i}
\end{equation}
where $\hat{\mathcal{C}}_i$ is the strict transform of $\mathcal{C}_i$ under~$\sigma$ and $E_{P,i}$, $i=1,...,r_P$ 
accounts for the exceptional components in $\sigma^{-1}(P)$. 
Given $h\in \cO_{X,P}$ the germ of a holomorphic function at $P$, we use the following standard notation 
$\mult_{E_{P,i}} \sigma^*h$ to describe the order of vanishing of the total transform of $h$ along $E_{P,i}$. 
Analogously, given $\omega_P$ the germ of a holomorphic $2$-form not vanishing at~$P$, we use the 
notation $\nu_{P,i}:=1+\mult_{E_{P,i}} \sigma^*\omega_P$.

An $n$-th cyclic covering ramified along $\cC$ induces an $n$-th cyclic cover ramified along the divisor 
$\sigma^*\cC$ in $Y$~\eqref{eq:ttC}. One can define the line bundles $\mathcal{L}^{(k)}=\cO_Y(L^{(k)})$ 
as in~\eqref{eq:Lkliso}, where $H$ is the pull-back in $Y$ of a projective line.
The announced description of $H^1(Y,\mathcal{L}^{(k)})$ admits the following form.

\begin{theorem}[\cite{Artal94}]
\label{thm:conucleo_liso}
For $0\leq k<n$, let
\begin{equation}
\label{eq:defsigmak}
\sigma_k:
H^0(\PP^2,\mathcal{O}_{\PP^2}(k-3))\longrightarrow
\bigoplus_{P\in\sing^*(\mathcal{C})}
\frac{\mathcal{O}_{\PP^2,P}}{\mathcal{J}_{\mathcal{C},P,\frac{k}{n}}}
\end{equation}
be the natural map where 
\[
\array{rcl}
\mathcal{J}_{\mathcal{C},P,\lambda}&:=&
\{h\in\mathcal{O}_{\PP^2,P}\mid\mult_{E_{P,i}}\sigma^*(h)\geq
\left\lfloor \lambda n_{P,i}\right\rfloor
-\nu_{P,i}+1\}\\
&=&
\{h\in\mathcal{O}_{\PP^2,P}\mid\mult_{E_{P,i}}\sigma^*(h)>
\lambda n_{P,i}-\nu_{P,i}\}.
\endarray
\]
\[
\dim H^1(Y,\mathcal{L}^{(k)})=\dim\coker\sigma_k.
\]
\end{theorem}
The second goal of this work is to generalize this theorem to the weighted projective plane.
In order to do so we present a different proof of Theorem~\ref{thm:conucleo_liso} that allows 
for an extension to the singular case. This new proof is somewhat more conceptual than the original 
one from~\cite{Artal94} and will be outlined here for exposition purposes and to help the reader 
have a basic idea of the strategy that will be used for the proof of the main 
Theorem~\ref{thm:conucleo_singular}. 

\begin{proof}[Proof of Theorem{\rm~\ref{thm:conucleo_liso}}]
The proof starts using the Riemann-Roch Theorem on the line bundles $\mathcal{L}^{(k)}=\cO_Y(L^{(k)})$:
\begin{equation} 
\label{eq:RRliso}
\array{c}
\chi(Y,\mathcal{L}^{(k)})=
\dim H^0(Y,\mathcal{L}^{(k)})
-\dim H^1(Y,\mathcal{L}^{(k)})+
\dim H^2(Y,\mathcal{L}^{(k)})\\
=\chi(Y,\mathcal{O}_Y)+
\dfrac{{L}^{(k)}\cdot({L}^{(k)}-K_Y)}{2}.
\endarray
\end{equation}

Note that $\chi(Y,\mathcal{O}_Y)=1$ since $Y$ is a rational surface.
It is not hard to check that $H^0(Y,\mathcal{L}^{(k)})$
is a subspace of $H^0(\PP^2,\mathcal{O}_{\PP^2}(-k))$ and hence 
$\dim H^0(Y,\mathcal{L}^{(k)})=0$.

Using Serre duality Theorem, we obtain that
\[
\dim H^2(Y,\mathcal{L}^{(k)})=
\dim H^0\left(Y,\omega_Y\otimes{\mathcal{L}^{(k)}}^{-1}\right).
\]
The inclusion of $H^0\left(Y,\omega_Y\otimes{\mathcal{L}^{(k)}}^{-1}\right)$
in $H^0(\PP^2,\mathcal{O}_{\PP^2}(k-3))$ can be understood as the kernel
of the map $\sigma_k$ defined in~\eqref{eq:defsigmak}.
Consider the right-hand term in~\eqref{eq:RRliso}:
\begin{gather*}
1+\frac{L^{(k)}\cdot(L^{(k)}-K_Y)}{2}=
1+\frac{k(k-3)}{2}+\\
\frac{1}{2}\sum_{P\in\sing^*(\mathcal{C})}
\left(\sum_{i=1}^{r_P}
\left\lfloor\frac{k n_{P,i}}{n}\right\rfloor
E_{P,i}
\right)
\cdot
\left(\sum_{i=1}^{r_P}
\left(\left\lfloor\frac{k n_{P,i}}{n}\right\rfloor-\nu_{P,i}+1\right)
E_{P,i}
\right)=\\
\dim H_0(\PP^2,\mathcal{O}_{\PP^2}(k-3))+
\sum_{P\in\sing^*(\mathcal{C})} \alpha_{\mathcal{C},P}^{(k)}.
\end{gather*}
Putting everything together,
\begin{equation}\label{eq:conucleo1}
\dim H^1(Y,\mathcal{L}^{(k)})=
\dim\coker\sigma_k-
\sum_{P\in\sing^*(\mathcal{C})} 
\left(\alpha_{\mathcal{C},P}^{(k)}
+\dim\frac{\mathcal{O}_{\PP^2,P}}{\mathcal{J}_{\mathcal{C},P,\frac{k}{n}}}
\right).
\end{equation}
Note that the term
\[
\beta_{\mathcal{C},P}^{(k)}:=\alpha_{\mathcal{C},P}^{(k)}
+\dim\frac{\mathcal{O}_{\PP^2,P}}{\mathcal{J}_{\mathcal{C},P,\frac{k}{n}}}
\]
only depends on the topological type of $(\mathcal{C},P)$ and the number $\frac{k}{n}$,
for $n=\deg \cC$. 
For $N$ big enough, we can find a projective plane curve $\mathcal{C}_N$ of degree~$N$
such that $\sing(\mathcal{C}_N)=\{P_N\}$ and whose local topological type (even the analytical one)
at $(\mathcal{C}_N,P_N)$ coincides with $(\mathcal{C},P)$.
A classical theorem by Nori extending Zariski's conjecture (see~\cite[Theorem II, p.~306]{Nori-zariski}) 
guarantees that, for $N$ big enough, the fundamental group of $\PP^2\setminus\mathcal{C}_N$ is abelian.
One can assume the extra condition that $N=Mn$ for some $M\gg 0$, hence~\eqref{eq:conucleo1} 
for $\mathcal{C}_N$ and $kM$ implies that
\[
\dim\coker\sigma_{kM}=\beta_{\mathcal{C}_N,P_{N}}^{(kM)}
=\beta_{\mathcal{C},P}^{(k)}.
\]
Again, for $M$ big enough the map $\sigma_{kM}$ is surjective and thus $\beta_{\mathcal{C},P}^{(k)}=0$.
\end{proof}

\section{Esnault--Viehweg's theorem revisited}\label{sc:Esnault}

Here we present a generalization of Esnault--Viehweg's results~\cite{Esnault-Viehweg82} about cyclic
branched covering assuming the ambient surfaces involved have abelian quotient singularities and
the divisors have simple $\Q$-normal crossings --\,this problem was also considered by 
Steenbrink~\cite[Lemma~3.14]{Steenbrink77} in the context of the semistable reduction associated with
an embedded resolution. Let us start with a couple of preliminary results.

\begin{lemma}\label{lemma:acyclic}
Let $\pi: Y \to X$ be a weighted blow-up at a cyclic quotient singular point
and let $D$ be a Weil divisor on $X$. Then $R^i \pi_{*} \cO_Y( \floor{\pi^{*}D} ) = 0$ for $i>0$.
\end{lemma}

\begin{proof}
For dimension reasons it is enough to prove the result for $i=1$.
Consider $\rho: Z \to Y$ a resolution of singularities of $Y$ and denote by $\sigma$ the composition
$\pi \circ \rho$ which is in turn a resolution of $X$. Let us denote $\cF := \cO_Z(\floor{\sigma^{*}D})$.
By~\cite[\S1.1]{Blache95}, one has $R^1 \sigma_{*} \cF = 0$.

On the other hand, Grothendieck's spectral sequence $E_2^{p,q} = (R^p \pi_* \circ R^q \rho_*) \cF
\Longrightarrow R^{p+q} \sigma_* \cF$ gives rise to the following
$$
R^1 \sigma_* \cF \simeq (R^0 \pi_* \circ R^1 \rho_*) \cF \oplus (R^1 \pi_* \circ R^0 \rho_*) \cF
$$
and, in particular, $(R^1 \pi_* \circ R^0 \rho_*) \cF = 0$.

Finally,
$$
R^0 \rho_* \cF = \rho_* \cF = \rho_* \cO_Z(\floor{\sigma^*D})
= \rho_* \cO_Z ( \floor{\rho^* \pi^* D} ) = \cO_Y ( \floor{\pi^*D} ),
$$
where the latest equality holds by the projection formula in Theorem~\ref{thm:projection}.
\end{proof}

\begin{prop}\label{prop:comparison_cohomology}
Under the assumptions of Lemma{\rm~\ref{lemma:acyclic}}, one has the isomorphisms
\[
H^p (X, \cO_X(D)) \simeq H^p (Y, \cO_Y( \floor{\pi^{*} D} ))
\text{ and }H^p (X, \cO_X(D)) \simeq H^p (Y, \cO_Y( \ceil{\pi^{*} D + K_\pi} )).
\]
\end{prop}

\begin{proof}
Since $\cO_Y(\floor{\pi^{*}D})$ is acyclic for the functor $\pi_{*}$ by Lemma~\ref{lemma:acyclic},
Leray's spectral sequence $E_2^{p,q} = H^p(X,R^q \pi_{*} \cO_Y(\floor{ \pi^{*} D }) ) \Longrightarrow
H^{p+q}(Y,\cO_Y(\floor{\pi^*D}))$ provides the isomorphism $H^p (X, R^0 \pi_{*} \cO_Y(\floor{\pi^{*}D})) \simeq
H^p(Y,\cO_Y(\floor{\pi^{*}D}))$. By the projection formula on normal surfaces~Theorem~\ref{thm:projection},
one obtains $\cO_X(D) = \pi_{*} \cO_Y(\floor{\pi^{*} D})$ and hence the first isomorphism holds.

The second isomorphism is a consequence of combining the first one with a generalization of Serre's duality
in this context~\cite[\S3.1]{Blache95}, as follows
\begin{align*}
& H^p(X,\cO_X(D)) = H^{2-p}(X,\cO_X(-D+K_X)) = H^{2-p}(Y,\cO_Y(\floor{\pi^{*}(-D+K_X)})) \\
&= H^p (Y, \cO_Y( - \floor{\pi^{*}(-D+K_X)} + K_Y)) = H^p (Y, \cO_Y( - \floor{ -\pi^{*} D + \pi^{*} K_X - K_Y } )) \\
&= H^p (Y, \cO_Y( - \floor{-\pi^*D - K_\pi} )) = H^p (Y, \cO_Y( \ceil{\pi^*D+K_\pi} )).
\end{align*}
Recall that $K_Y = \pi^{*} K_X + K_\pi$ and both canonical divisors are Weil divisors with integral coefficients.
\end{proof}

Let $\rho: \tilde{X} \to X$ be a cyclic branched covering of $n$ sheets between two projective surfaces having at most
abelian quotient singularities. In particular, $H^1(\tilde{X},\CC)$ is endowed with a pure Hodge structure compatible
with the monodromy of the covering. Due to the construction of this Hodge structure $H^1(\tilde{X},\CC)$ is
naturally isomorphic to $H^1(\tilde{X},\cO_{\tilde{X}}) \oplus H^0(\tilde{X},\Omega_{\tilde{X}}^1)$ where
$H^0(\tilde{X},\Omega_{\tilde{X}}^1)$ is the complex conjugation of $H^1(\tilde{X},\cO_{\tilde{X}})$.
The following result computes the Hodge structure of $H^q(\tilde{X},\cO_{\tilde{X}})$ providing
the decomposition in terms of its invariant subspaces.

\begin{theorem}\label{thm:Esnault}
Under the previous notation, assume the ramification set is given by a simple $\Q$-normal crossing integral
Weil divisor $D = \sum_{i=1}^r n_i D_i$ which is linearly equivalent to $nH$ for some integral
Weil divisor $H$. Then
$$
  H^q(\tilde{X},\cO_{\tilde{X}}) = \bigoplus_{k=0}^{n-1} H^q (X, \cO_X(L^{(k)})), \quad
  L^{(k)} = -k H + \sum_{i=1}^r \floor{\frac{kn_i}{n}} D_i,
$$
where the monodromy of the cyclic covering acts on $H^q(X,\cO_X(L^{(k)}))$ by multiplication by~$e^{\frac{2\pi ik}{n}}$.
\end{theorem}

\begin{proof}
Since the result is true for the smooth case~\cite{Steenbrink77, Esnault-Viehweg82}, it is enough to show
that the cohomology $H^q (X, \cO_X(L^{(k)}))$ remains the same after performing the first weighted blow-up in
Hirzebruch's resolution.

Let $\pi: Y \to X$ be the $(1,p)$-blow-up at a point of type 
$\frac{1}{d}(1,p)$ (see~\S\ref{sec:Qres} for the notation) and denote by $E$ its exceptional
divisor. Consider $D' = \pi^{*} D - n \pi^{*} H + n \ceil{\pi^{*} H}$ -- it is a Weil divisor
linearly equivalent to $n \ceil{\pi^{*} H}$ and, in addition, it is effective if $D$ is.
Denote by $L'^{(k)}$ its associated Weil divisor. There exist
$\rho_Y: \tilde{Y} \to Y$ another cyclic branched covering equivalent to $\rho_X:=\rho$ with ramification set given
by $D'$ and $\tilde{\pi}: \tilde{Y} \to \tilde{X}$ a proper map completing the following diagram:
\[
\begin{tikzcd}
&\tilde{X}\ar[d,"\rho_X" left] &\ar[l,"\tilde{\pi}" above]\tilde{Y}\ar[d,"\rho_Y"] &\\
L^{(k)},D\ar[r,hook] & X &\ar[l,"\pi"] Y \ar[r,hookleftarrow]& D',L'^{(k)}
\end{tikzcd}
\]
Locally, one can assume $D_1 = \{ x=0 \}$ and $D_2 = \{ y=0 \}$ with $m_1$ and $m_2$ possibly zero.
Denote by $c \in \frac{1}{d}\ZZ $ the rational number such that $\ceil{\pi^{*}H} - \pi^{*}H = c E$.
Note that $\pi^{*} D_1 = \hat D_1 + \frac{1}{d} E$ and 
$\pi^{*} D_2 = \hat D_2 + \frac{p}{d} E$. Then,
\begin{align*}
\pi^{*} L^{(k)} &= -k \pi^{*} H + \sum_{i=1}^r \floor{\frac{kn_i}{n}} \hat D_i
+ \left( \frac{1}{d} \floor{\frac{kn_1}{n}} + \frac{p}{d} \floor{\frac{kn_2}{n}} \right) E, \\
D' &= \sum_{i=1}^r n_i \hat D_i + \left( \frac{1}{d} n_1 + \frac{p}{d} n_2 + cn \right) E, \\
L'^{(k)} &= - k \ceil{\pi^{*}H} + \sum_{i=1}^r \floor{\frac{kn_i}{n}} \hat D_i
+ \floor{\frac{k \left(\frac{1}{d}n_1+\frac{p}{d}n_2+cn \right)}{n}} E.
\end{align*}

The $\QQ$-Weil divisor $L'^{(k)} - \pi^{*} L^{(k)}$ has support on $E$ and its coefficient is given by
$$
\gamma := -kc + \floor{ \frac{1}{d} \frac{kn_1}{n} + \frac{p}{d} \frac{kn_2}{n} + kc}
- \frac{1}{d} \floor{\frac{kn_1}{n}} - \frac{p}{d} \floor{\frac{kn_2}{n}} \in \frac{1}{d}\ZZ.
$$
Using the fact that $a-1 < \floor{a} \leq a, \forall a \in \QQ$, one shows that $\gamma$ belongs
to the open real interval 
$\left( -1, \frac{1+p}{d} \right)$.

Two cases arise. If $\gamma \leq 0$, then $\floor{\pi^{*}L^{(k)}} = L'^{(k)} + \ceil{\gamma} E = L'^{(k)}$. Otherwise,
since $K_\pi = (-1 + \frac{1+p}{d})E$, 
one has $\ceil{\pi^{*} L^{(k)} + K_\pi} = L'^{(k)} + \ceil{-\gamma-1+\frac{1+p}{d}}E = L'^{(k)}$.
In both cases, by Proposition~\ref{prop:comparison_cohomology}, $H^q(Y,\cO_Y(L'^{(k)})) \simeq H^q(X,\cO_X(L^{(k)}))$ and the claim
follows.
\end{proof}

\section{Local study: quasi-adjunction modules}
\label{sec:local}

The adjuntion ideals $\mathcal J_{\cC,P\frac{k}{n}}$ in Theorem~\ref{thm:conucleo_liso} for cyclic covers 
of $\PP^2$ ramified along reduced curves were generalized by Libgober~\cite{Libgober-characteristic} for 
abelian covers and reduced curves in $\PP^2$ introducing ideals of quasi-adjunction~\ref{rem:qaideals}.
Here we will introduce the analogous objects for the cyclic covers of $\PP^2_w$ ramified along non-reduced
curves. These objects are not ideals in general, but modules over the local rings $\cO_{\PP^2_w,P}$ which 
will be called quasi-adjunction modules.

Consider the local situation at $0\in X=\frac{1}{w_2}(w_0,w_1)\cong \CC^2/G$ a cyclic singularity, 
as defined in \S\ref{subsec:CyclicQuotient}, fixing $\zeta\in \CC$ a $w_2$-th primitive root of unity.
In \S\ref{subsec:CyclicQuotient} we defined the class $[g]:k\in\ZZ_{w_2}$ of a quasi-germ $g\in\cO_{X,\zeta^{k}}$ 
(which depends on the choice of~$\zeta$).
Also note that any $g\in \cO_{X,\zeta^{k}}$ defines a Weil divisor in $X$ denoted by $\ddiv_X(g)$.

Let $C=\sum_{i=1}^r n_iC_i$ be a local Weil divisor at $0$ given by the set of zeros of quasi-germs $C_i=\{f_i=0\}$.
Denote by $\pi: Y \to X$ a good $\Q$-resolution of $C\subset (X,0)$ and $\Gamma$ the dual graph associated with~$\pi$.
Then the total transform of $C$ and the relative canonical divisor can be written as
\begin{equation}\label{eq:notationEi-local}
\begin{aligned}
\pi^{*} C &= \hat{C} + \sum_{\v \in \Gamma} m_\v E_\v,
& \quad \pi^{*} C_i &= \hat{C}_i + \sum_{\v \in \Gamma} m_{\v i} E_\v, 
& \quad K_\pi &= \sum_{\v \in \Gamma} (\nu_\v-1) E_\v,
\end{aligned}
\end{equation}
where $m_\v=\sum_{i=1}^r n_im_{\v i}$.

\begin{defi}
\label{def:M}
For any integer $k\in \ZZ$ and rational $\lambda\in \QQ$ we introduce the following 
\emph{quasi-adjunction $\cO_X$-modules}:
\begin{equation}
\label{eq:Mlocal}
\cM_\pi(C^\lambda,k)=
\left\{ g \in \cO_{X,\zeta^{\s_{\lambda,k}}}
\vphantom{\sum_{i=1}^r}\right. 
\left|\ 
\mult_{E_\v} (\pi^{*}\ddiv_X(g)) > \sum_{i=1}^r \fpart{\lambda n_i} m_{\v i} - \nu_\v, \
\forall \v \in \Gamma
\right\}
\end{equation}
where $\s_{\lambda,k}:=k - |w| - \sum_{i=1}^r \floor{\lambda n_i} [f_i]$ and $|w|:=w_0+w_1+w_2$.
\end{defi}

In section~\ref{sec:examples} (Proposition~\ref{prop:indres}) it will be shown that $\cM_\pi(C^\lambda,k)$ 
does not depend on the given $\Q$-resolution. Also, the following upper semi-continuity property will be useful.

\begin{lemma}
\label{lemma:propsM}
Under the above conditions: 
\begin{enumerate}[label=\rm(\arabic*)]
 \item\label{lemma:propsM:epsilon}
 $\cM_\pi(C^\lambda,k)=\cM_\pi(C^{\lambda+\varepsilon},k)$ for a sufficiently small $\varepsilon>0$.
 \item\label{lemma:propsM:plus1} 
 $\cM_\pi(C^{\lambda+1},k)=\cM_\pi(C^{\lambda},k-[C])$, where $[C]$ is the class of $C$. 
 
 More generally, $\cM_\pi\left(C^{\lambda+\frac{1}{\gcd(n_i)}},k\right)=\cM_\pi\left(C^{\lambda},k-\left[\frac{1}{\gcd(n_i)}C\right]\right)$.

\end{enumerate}
\end{lemma}

\begin{proof}
First note that both are submodules of the same invariant module $\cO_{X,\zeta^{\s_{\lambda,k}}}$ since 
the class $\s_{\lambda,k}$ does not change for a small enough $\varepsilon>0$.
The ring $\cO_X$ is Noetherian, thus $\cM_\pi(C^\lambda,k)$ admits a finite system of generators. Since the conditions
in~\eqref{eq:Mlocal} 
$$
\mult_{E_\v} (\pi^{*}\ddiv_X(g)) > \sum_{i=1}^r \fpart{\lambda n_i} m_{\v i} - \nu_\v
$$
are given by a strict inequality there is a small positive $\varepsilon'>0$ for which the inequalities
$$
\mult_{E_\v} (\pi^{*}\ddiv_X(g)) > \sum_{i=1}^r \fpart{\lambda n_i} m_{\v i} - \nu_\v + \varepsilon'
$$
still hold. Taking $\varepsilon<\frac{\varepsilon'}{nmr}$ for $n:=\max\{n_j\}, m:=\max\{m_{\v j}\}$ such that
$$\fpart{(\lambda+\varepsilon) n_i} m_{\v i}< \fpart{\lambda n_i} m_{\v i} + \frac{\varepsilon'}{r}$$ 
the result follows.

As for part~\ref{lemma:propsM:plus1} note that $\s_{\lambda+1,k}=\s_{\lambda,k-[C]}$ and the set of conditions
in~\eqref{eq:Mlocal} does not change since $\left\{(\lambda+1)n_i\right\}=\left\{\lambda n_i\right\}$.
In general, the result follows from 
\begin{equation*}
\s_{\lambda+\frac{1}{\gcd(n_i)},k}=\s_{\lambda,k-\left[\frac{1}{\gcd(n_i)}C\right]}\text{ and } 
\left\{(\lambda+\frac{1}{\gcd(n_i)})n_i\right\}=\left\{\lambda n_i\right\}.
\qedhere
\end{equation*}
\end{proof}

\begin{remark}
Note that in the smooth case $X=(\CC^2,0)$ one can drop $k$ in the notation $\cM(C^\lambda,k)$ and simply write 
$\cM(C^\lambda)$, since both $\cO_{X,\zeta^{\s_{\lambda,k}}}=\cO_{X}$ and the conditions on $g \in \cO_{X}$
in~\eqref{eq:Mlocal} are independent of~$k$. Moreover, if $0\leq \lambda<\frac{1}{\max\{n_i\}}$, then $\cM(C^\lambda)$ 
coincides with the multiplier ideal associated with the principal ideal generated by an equation of $C$ (for a definition 
of multiplier ideals see for instance~\cite{Blickle-Lazarsfeld-informal}). This follows from the fact that 
$\left\{\lambda n_i\right\}=\lambda n_i$ for all $i=1,...,r$.
\end{remark}

\begin{remark}
\label{rem:qaideals}
Ideals of quasi-adjunction $A(j_1,\dots,j_r|m_1,\dots,m_r)$ ($0\leq j_i<m_i$) for $X=\CC^2$ were originally defined 
by A.Libgober in~\cite[section 2.3]{Libgober-characteristic}. Given a local divisor $C=\sum_{i=1}^r n_iC_i$ in 
$(\CC^2,0)$ and $\lambda=\frac{k}{d}$ such that $0\leq \lambda<\frac{1}{\max\{n_i\}}$, one can write $\cM(C^\lambda)$ as
a quasi-adjunction ideal choosing $m_i=d$, and $j_i+1=d(1-\lambda n_i)=d-kn_i$, that is,
$$
\cM(C^\lambda)=A(j_1,\dots,j_r|d,\dots,d).
$$
\end{remark}

\begin{remark}
\label{rem:independenciak}
In general, note that $\s_{\lambda,k}$ is independent of $\lambda$ as long as $0\leq \lambda<\frac{1}{\max\{n_i\}}$.
Moreover, in this case the expected monotonicity property holds
$$
\cO_{X,\zeta^{\s_{\lambda,k}}}=
\cO_{X,\zeta^{k-|w|}}\supseteq \cM(C^{\lambda_1},k)\supseteq \cM(C^{\lambda_2},k), \quad \text{ for } 
0\leq \lambda_1\leq \lambda_2<\frac{1}{\max\{n_i\}}.
$$
Example~\ref{ex:non-reduced} exhibits that this poperty does not need to hold for $\lambda\geq \frac{1}{\max\{n_i\}}$.
\end{remark}

\begin{ex}
Consider $X=\frac{1}{5}(2,3)$ and the irreducible Weil divisor $C$ defined as the zero set of $x^3-y^2$.
By Remark~\ref{rem:independenciak}, if $\lambda\in [0,1)$ one has a stratification for $\cM(C^\lambda,k)$
for the different $k=0,...,4$ as shown in Table~\ref{tab:MX}.

\begin{table}[ht]
\begin{center}
\begin{tabular}{|c|cccccccccccc|}
\hline
& $[0,\frac{1}{6})$ && $[\frac{1}{6},\frac{1}{3})$ && $[\frac{1}{3},\frac{1}{2})$ 
&& $[\frac{1}{2},\frac{2}{3})$ && $[\frac{2}{3},\frac{5}{6})$ && $[\frac{5}{6},1)$ &\\ \hline
0 & $\cO_X$ &$\supsetneq$ & $\langle xy,x^5,y^5\rangle$ & && && &&  & &\\ \hline
1 &  $\cO_{X,\zeta^1}$ && & && && && &&\\ \hline
2 &  $\cO_{X,\zeta^2}$ && && &$\supsetneq$ & $\langle y^4\rangle$ && && &\\ \hline
3 &  $\cO_{X,\zeta^3}$ && && && &$\supsetneq$ & $\langle x^4\rangle$ && &\\ \hline
4 &  $\cO_{X,\zeta^4}$ && && && && &$\supsetneq$ & $\langle y^3\rangle$ &\\ \hline
\end{tabular}
\end{center}
\caption{Stratification of $\cM(C^\lambda,k)$ for $\lambda\in [0,1)$}
\label{tab:MX}
\end{table}
\end{ex}

\begin{ex}
\label{ex:non-reduced}
As a second example we will consider a non-reduced curve in a smooth surface.
Take the divisor $\cC=\cC_1+2\cC_2$ in $(\CC^2,0)$, where $\cC_1$ is an ordinary cusp $\{y^2+x^3=0\}$, and $\cC_2$ is 
a smooth germ tangent to $\cC_1$, say $\{y=0\}$. Following the notation in~\eqref{eq:notationEi-local}, $\Gamma=\{\v \}$ 
since one weighted blow-up with $w=(2,3)$ resolves the curve. Note that $m_{\v 1}=6$, $m_{\v 2}=3$,
$n_{1}=1$, $n_2=2$, and $\nu_{\v}=5$, then according to Definition~\ref{def:M} one obtains the following condition
$$
\mult_{E_\v}\pi^*x^iy^j=2i+3j>6\{\lambda\}+3\{2\lambda\}-5.
$$
\begin{figure}[ht]
\begin{center}
\begin{tikzpicture}
\draw[line width=1.2] (-1,.5) node[above left] {$\mathcal{C}_1$} to[out=-60,in=180] (0,0) to[out=180,in=60] (-1,-.5);
\draw[line width=1.2] (-1,0)--(.5,0) node[right] {$\mathcal{C}_2$};

\draw (2.5,0)--(5.5,0) node[right] {$E_\nu$};
\fill (3,0) circle [radius=.1] node[below left] {$\frac{1}{2}(1,1)$};
\fill (5,0) circle [radius=.1] node[below] {$\frac{1}{3}(1,1)$};
\draw[line width=1.2] (3,-1)--(3,1) node[right] {$\mathcal{C}_2$};
\draw[line width=1.2] (4,-1)--(4,1) node[right] {$\mathcal{C}_1$};
\end{tikzpicture}
\caption{$\QQ$-Resolution of $\mathcal{C}$}
\end{center}
\end{figure}	
The jumping values for $\lambda\in [0,1)$ are provided below
$$
\cM(\cC^\lambda)=
\begin{cases}
\CC\{x,y\} & \text{ for } \lambda\in [0,\frac{5}{12}),\\
\mathfrak{m}=(x,y) & \text{ for } \lambda\in [\frac{5}{12},\frac{1}{2}),\\
\CC\{x,y\} & \text{ for } \lambda\in [\frac{1}{2},\frac{2}{3}),\\
\mathfrak{m} & \text{ for } \lambda\in [\frac{2}{3},\frac{5}{6}),\\
(x^2,y) & \text{ for } \lambda\in [\frac{5}{6},\frac{11}{12}),\\
\mathfrak{m}^2 & \text{ for } \lambda\in [\frac{11}{12},1).\\
\end{cases}
$$
As mentioned in Remark~\ref{rem:independenciak} the sequence $\cM(C^\lambda)$ is decreasing for
$\lambda\in [0,\frac{1}{2})$. However this is not true for $\lambda\in [0,1)$.
\end{ex}

\section{Global study: irregularity of cyclic coverings}
\label{sec:global}
Let $w_0,w_1,w_2$ be three pairwise coprime positive integers and let $X:=\PP_w^2$ be the weighted projective plane
with $w=(w_0,w_1,w_2)$. The section is split into three parts. First, an interpretation of $H^0(Y,\cO_Y(D'))$ for a 
blow-up $Y$ of $X$ and $\QQ$-divisors in terms of quasi-polynomials. Then the second cohomology groups of the 
cover are studied using Serre's duality. Finally a formula for the irregularity is described with the help of
the Euler characteristic and the Riemann-Roch formula on singular normal surfaces.

\subsection{Global sections and weighted blow-ups}
\mbox{}

Consider $\pi: Y \to X$ a composition of $s$ weighted blowing-ups with exceptional divisors $E_1,\ldots,E_s$.

\begin{lemma}\label{lemma:h0-weighted}
Given $D \in \Weil_{\QQ}(X)$, the space $H^0 (X,\!\cO_{X}(D))$ is isomorphic to $\CC[x,y,z]_{w,d}$,
the $w$-homogeneous polynomials of degree $d := \deg_w(\floor{D})$.
\end{lemma}

\begin{proof}
It is a well-known result for integral Weil divisors. The general rational case follows from the fact that
by definition $\cO_X(D) = \cO_X(\floor{D})$. The isomorphism converts $(h)$ into the $w$-homogeneous polynomial defined by
$(h) + \floor{D} \geq 0$.
\end{proof}

\begin{prop}\label{prop:H0YD}
For any $D \in \Weil_{\QQ}(X)$, $m_i \in \QQ$, and $D' \in \Weil_{\QQ}(Y)$, one has:
\begin{enumerate}[label=\rm(\arabic*)]
\item \label{item1-lemma-h0} Under the previous isomorphism $H^0(Y,\cO_Y(\pi^{*}D + \sum_{i=1}^{s} m_i E_i))$ can be identified
with the subspace $\{ F \in \CC[x,y,z]_{w,d} \mid \mult_{E_i} \pi^{*} F + m'_i \geq 0, \forall i \}$,
where $m'_i := m_i + \mult_{E_i} \pi^{*} \fpart{D}$, $i=1,\ldots,s$.
\item \label{item2-lemma-h0} The cohomology $H^0(Y,\cO_Y(D'))$ consists of all $w$-homogeneous polynomials~$F$ of degree
$\deg_w \floor{\pi_{*} D'}$ such that $\mult_{E_i} \pi^{*} F \geq \mult_{E_i} ( \pi^{*} \floor{\pi_{*} D'} - D' )$
for all $i=1,\ldots,s$.
\end{enumerate}
\end{prop}

\begin{proof}
The statement~\ref{item2-lemma-h0} is a direct consequence of~\ref{item1-lemma-h0} after
writing $D'$ in the form $\pi^{*}D + \sum_{i=1}^{s} m_i E_i$ for $D=\pi_{*} D'$ and $m_i = \mult_{E_i}(D'-\pi^{*} D)\in\QQ$.

In order to prove~\ref{item1-lemma-h0}, let us consider $h \in K(Y)$. According to Proposition~\ref{prop:H0D},
its associated divisor $(h)$ belongs to $H^0(Y,\cO_Y(\pi^{*}D + \sum_{i=1}^{s} m_i E_i))$ if and only if
\begin{equation}\label{eq:condition-h}
(h) + \pi^{*}D + \sum_{i} m_i E_i \geq 0.
\end{equation}
Since $K(Y)$ and $K(X)$ are isomorphic under $\pi$, the global section $h$ can be written as
$h = \pi^{*} h'$ for some $h' \in K(X)$. The condition from~\eqref{eq:condition-h}
in terms of $h'$ is $(h') + D \geq 0$, or equivalently $h \in H^0(X,\cO_X(D))$, and
\begin{equation}\label{eq:condition-h'}
\pi^{*} ((h')+D) + \sum_i m_i E_i = \pi^{*} ((h') + \floor{D}) + \pi^{*} \fpart{D} + \sum_{i} m_i E_i \geq 0.
\end{equation}
Let $F$ be the $w$-homogeneous polynomial defined by the effective $(h')+\floor{D}$.
Using the isomorphism $H^0(X,\cO_X(D)) \cong \CC[x,y,z]_{w,d}$ described in Lemma~\ref{lemma:h0-weighted},
the inequality~\eqref{eq:condition-h'} simply becomes $\mult_{E_i} \pi^{*} F + m'_i \geq 0$,
$\forall i=1,\ldots,s$, as claimed.
\end{proof}

\subsection{The second cohomology group}\label{sec:h2}
\mbox{}

Let $\rho_X: \tilde{X} \to \mathbb{P}^2_w =:X$ be a cyclic branched covering ramifying on a (non-necessarily) reduced curve
$\mathcal{C} = \sum_{j=1}^r n_j \mathcal{C}_j$ with $d:= \deg_w \mathcal{C}$ sheets, denote by $d_j = \deg_w \mathcal{C}_j$
so that $\sum_{j=1}^r n_j d_j = d$. Assume $\mathcal{C} \sim dH$ where $H$ is a divisor of $w$-degree one --it is neither
effective nor reduced in general.
In order to use the power of section~\ref{sc:Esnault}, one needs to deal with a $\QQ$-normal crossing divisor. Hence, let 
$\pi: Y \to X$ be an embedded $\Q$-resolution of $\mathcal{C}$ and consider the maps $\tilde{\pi}: \tilde{Y} \to \tilde{X}$ 
and $\rho_Y: \tilde{Y} \to Y$ completing the following commutative diagram.
\[
\begin{tikzcd}
\MySymb{dr}\tilde{X}\ar[d,"\rho_X" left] &\ar[l,"\tilde{\pi}" above]\tilde{Y}\ar[d,"\rho_Y"]\\
X&\ar[l,"\pi"] Y
\end{tikzcd}
\]
Denote by $S$ the points of $\mathbb{P}^2_w$ that 
have been blown up in the resolution $\pi: Y \to X$
and $\Gamma_P$ the dual graph associated with the resolution of $P \in \Si$.
Then the total transform of $\mathcal{C}$ and $H$ and the relative canonical divisor can be written as
\begin{equation}\label{eq:notationEi}
\begin{aligned}
\pi^{*} \mathcal{C} &= \hat{\mathcal{C}} + \sum_{P \in \Si} \sum_{\v \in \Gamma_P} m_\v E_\v,
& \quad \pi^{*} \mathcal{C}_j &= \hat{\mathcal{C}}_j + \sum_{P \in \Si} \sum_{\v \in \Gamma_P}
m_{\v j} E_\v, \\
\pi^{*} H &= \hat{H} + \sum_{P \in \Si} \sum_{\v\in\Gamma_P} \b_\v E_\v,
& K_\pi &= \sum_{P \in S} \sum_{\v \in \Gamma_P} (\nu_\v-1) E_\v.
\end{aligned}
\end{equation}
Clearly one has the relations $\hat{\mathcal{C}} = \sum_{j=1}^r n_j \hat{\mathcal{C}}_j$
and $m_\v = \sum_{j=1}^r n_j m_{\v j}$.
Finally consider the divisor $L^{(k)}$ corresponding to the covering $\rho_Y$, see Theorem~\ref{thm:Esnault}.

\begin{theorem}\label{thm:h2Lk}
The dual space $H^2(Y,\cO_{Y}(L^{(k)}))^{*}$ is isomorphic to the $\CC$-vector space
\[
\left\{ F \in \CC[x,y,z]_{w,s_k-|w|}
\vphantom{\sum_{j=1}^r}\right. 
\left|\ 
\mult_{E_\v} \pi^{*} F > \sum_{j=1}^r \fpart{\frac{kn_j}{d}} m_{\v j} - \nu_\v, \
\forall \v \in \Gamma_P, \, \forall P \in \Si
\right\} 
\]
where
$\CC[x,y,z]_{w,l}$ is the vector space of $w$-homogeneous polynomials of degree $l$, 
\[
s_k := \sum_{j=1}^r \fpart{\frac{kn_j}{d}} d_j \in \mathbb{Z},\text{ and }|w|=w_0+w_1+w_2.
\]
If $\mathcal{C}$ is reduced, then $s_k = k$ and $\displaystyle\sum_{j=1}^r \fpart{\frac{kn_j}{d}} m_{\v j} = \frac{km_\v}{d}$.
\end{theorem}

\begin{proof}
By Serre's duality $H^2(Y,\cO_{Y}(L^{(k)}))^{*} \cong H^0(Y,\cO_{Y}(K_Y-L^{(k)}))$.
We plan to apply Proposition~\ref{prop:H0YD} to the divisor $D' := K_Y - L^{(k)}$.
Recall that $K_Y = \pi^{*} K_X + K_{\pi}$.

According to~\eqref{eq:notationEi} the divisor $\pi^{*} \mathcal{C} - d \pi^{*} H$
is decomposed as
$$
- d \hat{H} + \sum_{j=1}^r n_j \hat{\mathcal{C}}_j + \sum_{P \in \Si}
\sum_{\v \in \Gamma_P} (m_\v - d \b_\v) E_\v
$$
and hence by definition $L^{(k)}$ is the divisor
\begin{equation*}
L^{(k)} = -k \hat{H} + \sum_{j=1}^r \floor{\frac{kn_j}{d}} \hat{\mathcal{C}}_j +
\sum_{P \in S} \sum_{\v \in \Gamma_P} \floor{\frac{k(m_\v - d\b_\v)}{d}} E_\v,
\end{equation*}
that can be easily written as
\begin{equation}\label{eq:Lk}
L^{(k)} = \pi^{*} \bigg( -kH + \sum_{j=1}^r \floor{\frac{kn_j}{d}} \mathcal{C}_j \bigg) +
\sum_{P \in S} \sum_{\v \in \Gamma_P} \bigg( \floor{\frac{k(m_\v-d\b_\v)}{d}} + k \b_\v - e_{\v k} \bigg) E_\v,
\end{equation}
where $e_{\v k} = \sum_{j=1}^r \floor{\frac{kn_j}{d}} m_{\v j}$.
Then 
\begin{equation}\label{eq:divisor-Ck}
\floor{\pi_{*} D'} = \floor{\pi_{*}(K_Y-L^{(k)})} = K_X + k H
- \sum_{j=1}^r \floor{\frac{k n_j}{d}} \mathcal{C}_j,
\end{equation}
which has $w$-degree
\begin{equation}\label{eq:degree-sk}
k - |w| - \sum_{j=1}^r \floor{\frac{kn_j}{d}} d_j = \sum_{j=1}^r \fpart{\frac{kn_j}{d}} d_j - |w|
= s_k - |w|.
\end{equation}
Note that if the curve $\mathcal{C}$ is reduced and thus all $n_i=1$, then $s_k = k$ and $e_{\v k} = 0$.
Moreover,
$$
\pi^{*} \floor{\pi_{*}D'} - D' = - K_\pi +
\sum_{P \in S} \sum_{\v \in \Gamma_P} \bigg( \floor{\frac{k(m_\v-d\b_\v)}{d}} + k \b_\v - e_{\v k} \bigg) E_\v.
$$

From Proposition~\ref{prop:H0YD}\ref{item2-lemma-h0}, $F \in H^0(Y,\cO_{Y}(K_Y-L^{(k)}))$ if and only if
for all $\v \in \Gamma_P$ and $P \in \Si$ one has
\begin{equation}\label{eq:condition-F}
\mult_{E_\v} \pi^{*} F \geq
\floor{\frac{k m_\v}{d}-k\b_\v} + k\b_\v - e_{\v k}
- (\nu_\v-1).
\end{equation}
It remains to show that this condition is equivalent to the one given in the statement.
If $F \in \CC[x,y,z]$ is $w$-homogeneous of degree $s_k-|w| = k - |w| - \sum_j \floor{\frac{kn_j}{d}} d_j$,
then the divisor $(F)-kH+K_X+\sum_{j=1}^r \floor{\frac{kn_j}{d}} \mathcal{C}_j$
has degree~$0$ and thus $\mult_{\v} \pi^{*} F - k\b_\v + \nu_\v-1 + e_{\v k}$ is an integer.
Hence the condition~\eqref{eq:condition-F} can be rewritten as
$$
\ceil{\mult_{E_\v} \pi^{*} F - \frac{km_\v}{d} + e_{\v k} + \nu_\v-1} \geq 0
$$
or equivalently $\mult_{E_\v} \pi^{*} F > \frac{km_\v}{d} - e_{\v k} - \nu_\v$.
Since $m_\v = \sum_{j=1}^r n_j m_{\v j}$ the latter term equals
$\sum_{j=1}^r \fpart{\frac{kn_j}{d}} m_{\v j}$ and the proof is complete.
\end{proof}

\subsection{Irregularity of a cyclic cover of a weighted projective plane}
\mbox{}

Let $\rho_X:\tilde X\to X$ be a cyclic covering and $\pi:Y\to X$ a $\Q$-resolution as defined in section~\ref{sec:h2}.
Recall that
$$
H^1(\tilde X,\CC)=H^1(\tilde Y,\CC)=H^1(Y,\cO_Y)\oplus H^0(Y,\Omega^1_Y)=H^1(Y,\cO_Y)\oplus \overline{H^1(Y,\cO_Y)}.
$$
In this section the dimension of the $e^{\frac{2\pi ik}{d}}$-eigenspace of $H^1(Y,\mathcal{O}_Y)$ will be computed
using the tools developed in section~\ref{sc:Esnault}. Hence the classical formula (see~Theorem~\ref{thm:conucleo_liso}) 
will be generalized in three directions: first, it applies to a surface with quotient singularities 
(the weighted projective plane); second, the ramification locus need not be completely 
resolved (a partial resolution is enough); and finally, the result also applies to non-reduced ramification divisors. For 
this latter purpose we define the following divisor $\mathcal{C}^{(k)}$ which is trivial when the ramification divisor is 
reduced:
$$
\mathcal{C}^{(k)} := \sum_{j=1}^r \floor{\frac{kn_j}{d}} \mathcal{C}_j.
$$
The divisor $\cC^{(k)}$ appeared implicitly in~\eqref{eq:divisor-Ck}.
Recall there is a natural isomorphism between 
$H^0\left(\PP^2_w,\mathcal{O}_{\PP^2_w}\left( kH+K_{\PP^2_w} - \mathcal{C}^{(k)}\right) \right)$
and the vector space of $w$-homogeneous polynomials of total degree $s_k - |w|$ (see~\eqref{eq:degree-sk} and statement of
Theorem~\ref{thm:h2Lk}).

\begin{theorem}\label{thm:conucleo_singular}
For $0\leq k<d$, let
\[
\pi^{(k)}: H^0\left(\PP^2_w,\mathcal{O}_{\PP^2_w}\left( kH+K_{\PP^2_w} - \mathcal{C}^{(k)}\right) \right) 
\longrightarrow \bigoplus_{P \in S}
\frac{\mathcal{O}_{\PP^2_w,P}\left( kH+K_{\PP^2_w} - \mathcal{C}^{(k)}\right)}{\mathcal{M}_{\mathcal{C},P}^{(k)}}
\]
be the evaluation map where $\mathcal{M}_{\mathcal{C},P}^{(k)}:=\mathcal{M}_P(\cC^{1/d},k)$ 
is defined as the following quasi-adjunction $\mathcal{O}_{\PP^2_w,P}$-module
\[
\mathcal{M}_{\mathcal{C},P}^{(k)}\!=\!
\left\{ g \in\mathcal{O}_{\PP^2_w,P}\left( kH+K_{\PP^2_w} - \mathcal{C}^{(k)}\right)
\vphantom{\sum_{j=1}^r}\!\right.\!
\left|\ \mult_{E_\v} \pi^* g > 
\sum_{j=1}^r \fpart{\frac{kn_j}{d}} m_{\v j} -\! \nu_\v, \ \forall \v \in \Gamma_P\! \right\}\!.
\]
Then,
\[
\dim H^1(Y,\mathcal{O}_Y(L^{(k)}))=\dim\coker\pi^{(k)}.
\]
In particular, if $\cC$ is reduced, then
\[
\pi^{(k)}: H^0\left(\PP^2_w,\mathcal{O}_{\PP^2_w}\left( kH+K_{\PP^2_w}\right) \right) 
\longrightarrow \bigoplus_{P \in S}
\frac{\mathcal{O}_{\PP^2_w,P}\left( kH+K_{\PP^2_w}\right)}{\mathcal{M}_{\mathcal{C},P}^{(k)}}
\]
where
$$
\mathcal{M}_{\mathcal{C},P}^{(k)}:=
\left\{ g \in\mathcal{O}_{\PP^2_w,P}\left( kH+K_{\PP^2_w}\right)
\vphantom{\frac{k m_{\v}}{d}}\right.
\left|\ \mult_{E_\v} \pi^* g > 
\frac{k m_{\v}}{d} - \nu_\v, \ \forall \v \in \Gamma_P \right\}.
$$
\end{theorem}

\begin{proof}
First note that by Theorem~\ref{thm:h2Lk} the kernel of $\pi^{(k)}$ is isomorphic to the dual of $H^2(Y,\cO_Y(L^{(k)}))$.
Rewriting $\dim\ker\pi^{(k)}$ in terms of $\dim\coker\pi^{(k)}$ and using the Euler characteristic
$\chi(Y,\cO_Y(L^{(k)})) = \sum_{q=0}^2 (-1)^q \dim H^q(Y,\cO_Y(L^{(k)}))$, one obtains
$$
\dim H^1(Y,\cO_Y(L^{(k)})) = \dim\coker\pi^{(k)} + A,
$$
where
\begin{equation}\label{eq:A}
\begin{aligned}
A &= \dim H^0\left(\PP^2_w,\mathcal{O}_{\PP^2_w}\left( kH+K_{\PP^2_w} - \mathcal{C}^{(k)}\right)\right)
- \sum_{P \in S} \dim \frac{\mathcal{O}_{\PP^2_w,P}\left( kH+K_{\PP^2_w} - \mathcal{C}^{(k)}\right)}{\mathcal{M}_{\mathcal{C},P}^{(k)}} \\
& \quad - \chi(Y,\cO_Y(L^{(k)})) + \dim H^0(Y,\cO_Y(L^{(k)})).
\end{aligned}
\end{equation}

For $k=0$, $\dim H^0(Y,\cO_Y(L^{(k)})) = \chi (Y,\cO_Y(L^{(k)})) = 1$ and the rest of the terms appearing in $A$
are zero. Then $A = 0$ and the claim follows.

Assume from now on that $k\neq 0$, then $\dim H^0(Y,\cO_Y(L^{(k)})) = 0$.
The first part of the proof is to rewrite $A$ in terms of local properties of the curve $\mathcal{C}$.
To compute the zero cohomologies in~\eqref{eq:A} we use the Riemann-Roch formula on normal surfaces (see~Theorem~\ref{thm:RR}).

Recall from~\eqref{eq:degree-sk} that $s_k - |w|$ is the degree of $kH+K_{\PP^2_w} - \mathcal{C}^{(k)}$.
Then since $s_k>0$, one obtains
\begin{equation}\label{eq:H0skw}
\begin{aligned}
h^0(\PP^2_w, & \mathcal{O}_{\PP^2_w}(s_k-|w|))\\
&=h^0\Big(\PP^2_w, \mathcal{O}_{\PP^2_w}\left( kH+K_{\PP^2_w} - \mathcal{C}^{(k)} \right)\Big) =
\chi \left( \PP^2_w,\mathcal{O}_{\PP^2_w} \left( kH+K_{\PP^2_w} - \mathcal{C}^{(k)} \right) \right) \\
& =1 + \frac{1}{2} \left( kH+K_{\PP^2_w} - \mathcal{C}^{(k)} \right)
\cdot \left( kH - \mathcal{C}^{(k)} \right)
+ R_{\PP^2_w} \left( kH+K_{\PP^2_w} - \mathcal{C}^{(k)} \right).
\end{aligned}
\end{equation}
After applying Serre's duality to $R_{\PP^2_w}$ the last term can be replaced:
\begin{equation}\label{eq:RXlocal}
R_{\PP^2_w} \left( kH+K_{\PP^2_w} - \mathcal{C}^{(k)} \right)=
R_{\PP^2_w} \left( -kH + \mathcal{C}^{(k)} \right) =
\sum_{P \in S} R_{{\PP^2_w},P} \left( -kH + \mathcal{C}^{(k)} \right).
\end{equation}

Now we are interested in studying the Euler characteristic of $\cO_Y(L^{(k)})$. From the proof
of Theorem~\ref{thm:h2Lk}, see~\eqref{eq:Lk}, it follows that
$$
L^{(k)} = \pi^{*} \left( -kH + \mathcal{C}^{(k)} \right)
+ \sum_{P \in S} \sum_{\v\in\Gamma_P} \left( \floor{\frac{k(m_\v-d\b_\v)}{d}} + k \b_\v
- e_{\v k} \right) E_\v.
$$
Let us denote by $\alpha_{\cC,P}^{(k)}$ the rational number
\begin{small}
\begin{equation}\label{eq:def-alpha}
\begin{aligned}
& \alpha_{\cC,P}^{(k)} =\\
&\quad \frac{1}{2} \sum_{\v,\vv \in \Gamma_P}
\left( \floor{\frac{k(m_\v-d\b_\v)}{d}} + k \b_\v - e_{\v k} \right)
\left( \floor{\frac{k(m_\vv-d\b_\vv)}{d}} + k \b_\vv - e_{\vv k} - (\nu_\vv - 1) \right)
E_\v \cdot E_\vv.
\end{aligned}
\end{equation}
\end{small}
Applying the Riemann-Roch formula (Theorem~\ref{thm:RR}) to the $V$-surface $Y$ one has
\begin{equation}\label{eq:chiLk}
\begin{aligned}
& \chi(Y,\cO_Y(L^{(k)})) = 1 + \frac{L^{(k)} \cdot (L^{(k)}-K_Y)}{2} + R_Y(L^{(k)}) \\
&= 1 + \frac{1}{2} \left( -kH + \mathcal{C}^{(k)} \right)
\cdot \left( -kH + \mathcal{C}^{(k)} - K_{\PP^2_w} \right)
+ \sum_{P \in S} \alpha_{\cC,P}^{(k)} + R_Y(L^{(k)}).
\end{aligned}
\end{equation}
Rearranging the local contributions to the correction term $R_Y(L^{(k)})$ one obtains
\begin{equation}\label{eq:RYlocal}
R_Y(L^{(k)}) = \sum_{Q \in Y} R_{Y,Q} (L^{(k)}) = \sum_{P \in S} \sum_{Q \in \pi^{-1}(P)} R_{Y,Q}(L^{(k)}).
\end{equation}

Putting all the previous calculations~\eqref{eq:A}, \eqref{eq:H0skw}, \eqref{eq:RXlocal},
\eqref{eq:chiLk}, \eqref{eq:RYlocal} together, one obtains $A = - \sum_{P \in S} \beta_{\cC,P}^{(k)}$
where
\begin{equation}\label{eq:beta}
\beta_{\cC,P}^{(k)}\! =\! \alpha_{\cC,P}^{(k)}
+ \dim\! \frac{\mathcal{O}_{\PP^2_w,P}\left( kH\! +\! K_{\PP^2_w}\! -\! \mathcal{C}^{(k)} \right)}{\mathcal{M}_{\mathcal{C},P}^{(k)}}
+ \!\!\!\!\sum_{Q \in \pi^{-1}(P)}\!\!\!\!\!\! R_{Y,Q}( L^{(k)} )
- R_{{\PP^2_w},P}\! \left(\! -kH \!+ \!\mathcal{C}^{(k)}\! \right)\!,
\end{equation}
and, finally,
\begin{equation}
\label{eq:H1-coker-beta}
\dim H^1(Y,\cO_Y(L^{(k)})) = \dim\coker\pi^{(k)} - \sum_{P \in S} \beta_{\cC,P}^{(k)}.
\end{equation}

The second part of the proof consists in showing that $\beta_{\cC,P}^{(k)}=0$ for any $P\in S$.
Without loss of generality one can assume $P=[0:0:1]$. The proof is analogous for the other singular points of $\PP^2_w$.
For the remaining points in $S$ the same proof works changing $w_2$ by 1.
\begin{step}\label{step1} 
 Note that $\beta_{\cC,P}^{(k)}$ only depends on the topological type of 
 $(\mathcal{C},P) \subset (\PP^2_w,P) = \frac{1}{w_2}(w_0,w_1)$, $\frac{k}{d}\in\QQ$, and $k$ (resp. $d_1,\ldots,d_r,|w|$) 
 modulo $w_2$.
\end{step}
 The result is trivial for the last summand of~\eqref{eq:beta}.  
 The key point for the first two summands is to show that the term $\floor{\frac{k m_\v}{d}-k\b_\v}+k\b_\v$ only depends on 
 $\frac{k}{d}$, $m_\v$, $\b_\v$ (which depend on the local topological type both of the surface and the curve), 
 and $k$ modulo $w_2$ as opposed to $k$. The latter is a consequence of the fact that $\b_\v\in \frac{1}{w_2}\ZZ$.
 Finally, for the third summand in~\eqref{eq:beta} one needs to use~\eqref{eq:Lk}.
 
\begin{step} For any given $P\in S$ we can apply Lemma~\ref{lemma:global-realization} to $(\cC,P)$ and obtain a global 
generic curve $\cD\subset \PP^2_{w}$ such that $(\cD,P)=(\cC,P)$. 
\end{step} 
\begin{step} 
 One applies~\eqref{eq:H1-coker-beta} to the new curve $\cD$. Note that the only singularity contributing to the 
 right-hand side is $P\in \cD$ and $(\cD,P)$ has the same topological type as $(\cC,P)$. 
 For a curve $\cD$ of big enough degree, the cokernel in~\eqref{eq:H1-coker-beta} is zero. In addition, 
 by~\cite[Theorem II, p.~306]{Nori-zariski} the left-hand side of~\eqref{eq:H1-coker-beta} is also zero and hence 
 $\beta_{\cD,P}^{(k')}=0$ for all $0\leq k'< \deg \cD$.
\end{step}
  
\begin{step}
Finally, combining properties~\ref{prop:4}, \ref{prop:5}, and \ref{prop:6} in Lemma~\ref{lemma:global-realization} 
together with Step~\ref{step1} above, it immediately follows that
$\beta_{\cC,P}^{(k)}=\beta_{\cD,P}^{(k')}$, for $k'=k\frac{\deg \cD}{\deg \cC}\in \ZZ$, since $k'\equiv k \mod (w_2)$,
$\deg \cD_j\equiv \deg \cC_j=d_j \mod (w_2)$ for all $j=1,...,r$, and $\frac{k}{\deg \cC}=\frac{k'}{\deg \cD}$.\qedhere
\end{step}
\end{proof}

\begin{ex}
As a first example we can study the $6$-fold cyclic cover $\tilde X$ of $\PP^2$ ramified along the non-reduced divisor 
$\cC=\cC_1+2\cC_2+3\cC_3$, where $\cC_1,\cC_2,\cC_3$ are any three concurrent lines. Note that 
$$
\cM(\cC^{\frac{k}{6}})=
\begin{cases}
\mathfrak{m}=(x,y) & \text{ for } k=5\\
\CC\{x,y\} & \text{ otherwise.}\\
\end{cases}
$$
By Theorem~\ref{thm:conucleo_singular} one needs to consider the morphism:
$$
\pi^{(k)}=
H^0(\PP^2,\cO_{\PP^2}(s_k-3)) \to \frac{\cO_{\CC^2,0}}{\cM(C^{\frac{k}{6}})}=
\begin{cases}
0 & \text{ if } k=0,...,4\\
\CC & \text{ if } k=5.\\
\end{cases}
$$
It is also straightforward to check that $s_5=2$. 
Hence $\coker \pi^{(5)}=\CC$ is the only non-trivial cokernel and thus $\dim H^1(\tilde X,\CC)=2$.
\end{ex}

\begin{lemma}[Global realization Lemma]
\label{lemma:global-realization}
Let $(\cC,P)$ be the topological type of a curve singularity in a cyclic quotient surface of normalized type 
$\frac{1}{w_2}(\bar w_0,\bar w_1)$. Then there exists a global curve $\cD\subset \PP^2_w$, $w=(w_0,w_1,w_2)$,
in a weighted projective plane such that:
\begin{enumerate}[label=\rm(\arabic*)]
 \item\label{prop:1} 
 $(\PP^2_w,O=[0:0:1])$ is a surface singularity of type $\frac{1}{w_2}(\bar w_0,\bar w_1)$,
 \item\label{prop:2} 
 $(\cC,P)$ and $(\cD,O)$ have the same topological type,
 \item\label{prop:3} One obtains
 $\cD\cap \sing \PP^2_w \subseteq\sing \cD=\{O\}$.
\end{enumerate}
Moreover, if $\cC=\cC_1\cup...\cup\cC_r$ is already a global curve in some weighted projective plane $\PP^2_w$ and 
$P\in \cC$ is as above with $(\PP,P)=\frac{1}{w_2}(\bar w_0,\bar w_1)$, then for any integer $N$ a curve 
$\cD=\cD_1\cup...\cup\cD_r$ can be found in $\PP^2_w$ satisfying the following 
extra conditions:
\begin{enumerate}[label=\rm(\arabic*)]
\setcounter{enumi}{3}
 \item\label{prop:3b}
 $\cD\cap \sing \PP^2_w \subseteq\{O\}$ and $\sing \cD=\{O\}\cup \cN$, 
 where $\cN$ is a set of nodal points of $\cD$ (at smooth points of $\PP^2_w$),
 \item\label{prop:4} 
 $\deg \cC$ divides $\deg_w \cD > N$,
 \item\label{prop:5} 
 $\deg \cD_j\equiv \deg \cC_j \mod (w_2)$, and
 \item\label{prop:6}
 $q \equiv 1 \mod (w_2)$, where $q:=\frac{\deg_w \cD}{\deg_w \cC}\in \ZZ_{> 0}$.
\end{enumerate}
\end{lemma}

\begin{proof}
Similarly as in the case of germs on a smooth surface, there is a polynomial representative $f(x,y)$ of $(\cC,P)$. 
In this case $f(x,y)=\sum_{(i,j)\in \Gamma(f)}a_{i,j}x^iy^j$ where 
\begin{equation}
\label{eq:ij}
\bar w_0 i + \bar w_1 j \equiv w_0 i + w_1 j\equiv \lambda \mod (w_2) 
\end{equation}
$(i,j)\in \Gamma(f)$ and the set $\Gamma(f)=\{(i,j)\in \ZZ^2\mid a_{i,j}\neq 0\}$ is finite and $\lambda$ is a fixed 
integer $0\leq \lambda<w_2$. Adding monomials of high enough $(w_0,w_1)$-degree satisfying~\eqref{eq:ij}, one can assure 
that the topological type of $\{g=f+F=0\}$ at $P=(0,0)$ coincides with that of $(\cC,P)$. The next step will be to 
homogenize $g$ with respect to $w=(w_0,w_1,w_2)$. Denote by $M:=\max\{w_0 i+w_1 j \mid (i,j)\in \Gamma(g)\}$.
Given any $(i,j)\in \Gamma(g)$ note that $M-(w_0 i+w_1 j) = d_{i,j}w_2\geq 0$. Hence, 
$$
G(X,Y,Z)=\sum_{(i,j)\in \Gamma(g)} a_{i,j} X^iY^jZ^{d_{i,j}}
$$
is a $w$-homogeneous polynomial of degree $M$.
For a generic choice of monomials in $F$, the global curve $\cD=\{G=0\}$ satisfies the required conditions
\ref{prop:1}-\ref{prop:3}.

For the second part, we proceed as before for each irreducible component $\cC_j$ obtaining $\cD=\cD_1\cup ...\cup\cD_r$.
By the generic choice of monomials for each irreducible component, condition~\ref{prop:3b} is immediately satisfied and it will 
avoid the other singular points of $\PP^2_w$ if $M_j=\deg_w \cD_j$ is a multiple of $w_0w_1$.
Also, note that $M_j=\deg_w \cD_j$ from the construction above and $M=\sum_j n_jM_j=\deg_w \cD$. Consider $\Delta\in \ZZ_{>0}$
such that $\Delta\equiv -w_2^{-1} \mod (w_0w_1)$ and define 
\begin{equation}
\label{eq:Mj}
M_j:=(\Delta w_2+1)\deg_w \cC_j,\quad \text{ for all } j=1,..,r,
\end{equation}
then $M_j$ satisfies~\ref{prop:5}. By construction 
$M=(\Delta w_2+1)\deg_w \cC $ which proves~\ref{prop:4} and $\frac{M}{\deg_w \cC}=\Delta w_2+1$ satisfies~\ref{prop:6}.
Since $\Delta$ can be chosen big enough, then the result follows.
\end{proof}

\section{Examples and applications}
\label{sec:examples}

In this section some applications of the global theory (Theorem~\ref{thm:conucleo_singular}) are provided. Some of them 
are theoretical, such as the independence of $\cM_\pi(C^\lambda,k)$ on the $\Q$-resolution, others are 
concrete examples of calculations. In particular, we provide examples that exhibit new features in our approach.
Namely, only a $\Q$-resolution is required to obtain the birational information on the ramified covers;
also the special relevance of the singular points of the surface is shown in a particular case; finally,
a Zariski pair in a weighted projective plane is described (see~\cite{ACM-Nemethi60} for more examples). 

\subsection{Independence of the \texorpdfstring{$\Q$}{Q}-resolution: global to local}
\mbox{}

As announced in section~\ref{sec:local} we can show that 
$\cM_\pi(C^\lambda,k)$ is independent of the chosen $\Q$-resolution. 
This is an interesting application where the global theory is used to prove a local result.

\begin{prop}
\label{prop:indres}
Let $\pi_i:Y_i\to X$, $i=1,2$ be two good $\Q$-resolutions of $C\subset (X,0)$, then
$$\cM_{\pi_1}(C^\lambda,k)= \cM_{\pi_2}(C^\lambda,k).$$
\end{prop}

\begin{proof}
By Lemma~\ref{lemma:propsM}\ref{lemma:propsM:plus1} it is enough to show the result for all $0\leq \lambda <1$.
Let us decompose $C=\sum n_iC_i$ where $C_i$ are the irreducible components of $C$.
The surface can be written as $X=\frac{1}{w_2}(w_0,w_1)$ a cyclic singularity in a normal form, that is, $\gcd(w_i,w_j)=1$ for $i\neq j$. 
Consider $0\leq \lambda<1$ a rational number and $0\leq k<w_2$. By Lemma~\ref{lemma:propsM}\ref{lemma:propsM:epsilon}, for any given good 
resolution $\pi:Y\to X$ of $(X,C)$ there is a $\varepsilon > 0$ such that $\cM_\pi(C^{\lambda},k)=\cM_\pi(C^{\lambda+\varepsilon},k)$. 
We claim one can find a projective curve $\cD\subset \PP^2_w$, $w=(w_0,w_1,w_2)$ as in~\eqref{eq:Mj} in the proof of 
Lemma~\ref{lemma:global-realization}, where
$\deg_w\cD=c(1+\Delta w_2)$ for $c>0$, $c\equiv [C] \mod (w_2)$ and a big enough 
$\Delta\gg 0$ and $a\equiv k \mod (w_2)$ such that 
$\lambda':=\frac{a}{\deg_w\cD}\in [\lambda,\lambda+\varepsilon)$.
Since $\cD$ is normal crossing outside of $P$ by property~\ref{prop:3b} 
in Lemma~\ref{lemma:global-realization} the morphism $\pi$ is also a resolution of $\cD$. Hence
$$
\cM_\pi(C^\lambda,k)=\cM_\pi(C^{\lambda'},k)=\cM_\pi(\cD^{\lambda'},k)=\cM_\pi(\cD^{\lambda'},a).
$$
The map $\pi^{(a)}$ in Theorem~\ref{thm:conucleo_singular} is as simple as
\[
\pi^{(a)}: H^0\left(\PP^2_w,\mathcal{O}_{\PP^2_w}\left( aH+K_{\PP^2_w} - \mathcal{D}^{(a)}\right) \right) 
\longrightarrow 
\frac{\mathcal{O}_{\PP^2_w,P}\left( aH+K_{\PP^2_w} - \mathcal{D}^{(a)}\right)}{\cM_\pi(\cD^{\lambda'},a)}.
\]
Since $\deg_w\cD$ can be chosen big enough, one can assume $\pi^{(a)}$ is surjective. Since 
$\ker \pi^{(a)}=H^2(Y,\cO_Y(L^{(a)}))$ is a birational invariant of the associated covering of $\PP^2_w$ ramified along $\cD$, 
one obtains
$$\dim_\CC \frac{\mathcal{O}_{\PP^2_w,P}\left( aH+K_{\PP^2_w} - \mathcal{D}^{(a)}\right)}{\cM_\pi(\cD^{\lambda'},a)}
=\dim_\CC \frac{\mathcal{O}_{\PP^2_w,P}\left( kH+K_{\PP^2_w} - \mathcal{D}^{(k)}\right)}{\cM_\pi(\cD^{\lambda'},k)}$$
is independent of the resolution.

Finally, the result follows since any two resolutions are simultaneously dominated by a third one 
and if $\pi$ dominates $\pi'$, then $\cM_{\pi}(C^\lambda,k)\subset \cM_{\pi'}(C^\lambda,k)$.

As for the claim, note that the proof of Lemma~\ref{lemma:global-realization} can be redone from the local situation 
using any list of positive $c_i\equiv [C_i] \mod (w_2)$ and $\deg_wD_i=c_i(1+\Delta w_2)$. Hence 
$\deg_wD=c:=\sum_i n_ic_i\equiv [C] \mod (w_2)$. Take any $\Delta\gg 0$ such that $\Delta\equiv -w_2^{-1} \mod (w_0w_1)$
and $\frac{w_2}{c(1+\Delta w_2)}<\varepsilon$, denote $N:=c(1+\Delta w_2)>0$. Take $n_1\in \ZZ_{>0}$ such that
$\frac{n_1}{N}\in [\lambda,\lambda+\varepsilon)$. Note that $\frac{n_1+j}{N}\in [\lambda,\lambda+\varepsilon)$ for 
all $j=0,...,w_2-1$, hence there is exactly one such $j$ for which $a:=n_1+j\equiv k \mod (w_2)$. This proves the claim.
\end{proof}

\subsection{Partial versus total resolutions}
\mbox{}

This example illustrates the advantage of using $\Q$-resolutions. In addition, this can be used to obtain information
about local invariants of the embedding of a curve in a cyclic quotient singularity.

\begin{ex}
Let $\mathcal{C}$ be the curve in $X:=\PP^2_{(9,5,2)}$ defined by $g = x^2 z + y^2 z^5 + y^4$.
Consider the cyclic branched covering $\rho: \tilde{X} \to X$ of degree $d = 20$ ramifying on $\mathcal{C}$.
This curve passes through the point $[1:0:0]$ which is of type $\frac{1}{9}(5,2)$ and it is quasi-smooth,
since it is locally isomorphic to $z=0$. The curve is singular at $[0:0:1] \in \frac{1}{2}(1,1)$ and it
defines a \emph{quasi-node}, i.e.~it is locally isomorphic to $xy=0$. The rest of the points are smooth and
$[0:1:0] \not \in \mathcal{C}$. Therefore the curve itself is already a $\mathbb{Q}$-divisor with (non-simple!) 
$\Q$-normal crossings and Theorem~\ref{thm:conucleo_singular} tells us that $H^1(\tilde{X},\CC)=0$ because the 
map $\pi^{(k)}$ is identically zero. Note that showing this fact is not obvious if one uses a standard good resolution.

We can use this fact to obtain a formula for the correction term map 
$$R_{X,P}: \Weil(X,P)/\Cart(X,P) \longrightarrow \QQ,$$
--see section~\ref{sec:RR}.
We resolve the singularity of $X$ at the point $P=[0:1:0] \in \frac{1}{5}(1,3)$ so that 
$\sum_{Q \in \pi^{-1}(P)} R_{Y,Q}( L^{(k)} ) = 0$.
Since the curve does not pass through this point, the second term in~\eqref{eq:beta} is zero and
the formula $\beta_{\mathcal{C},P}^{(k)} = 0$ becomes
$$
R_{X,P}(-kH) = \alpha_{\mathcal{C},P}^{(k)}
:= \frac{1}{2} \sum_{i,j=1}^2 \fpart{-k \b_i} \left( \fpart{-k \b_i} + \nu_i -1 \right) E_i \cdot E_j.
$$
Following the Hirzebruch-Jung method, one obtains two exceptional divisors $E_1$ and $E_2$ and the following
numerical data:
$$
A = \begin{pmatrix} -2 & 1 \\ 1 & -3 \end{pmatrix},
\quad \nu_1 = \dfrac{4}{5}, \quad \nu_2 = \dfrac{3}{5},
\quad \b_1 = -\dfrac{11}{5}, \quad \b_2 = -\dfrac{2}{5},
$$
where $A$ is the intersection matrix of the resolution and $H=\ddiv_X(x)-4\ddiv_X(z)$ was chosen for the computation
of $\b_1$ and $\b_2$. This completely describes the map $R_{X,P}$,
since $H$ generates the group $\Weil(X,P)/\Cart(X,P)$ --\,for instance $R_{X,P}(-3H) = -1/5$,
cf.~\cite[p.~312]{Blache95}.
\end{ex}

\subsection{Smooth versus singular surfaces}
\mbox{}

The next two examples illustrate the relevance of the codimension 2 singular locus of a surface when it comes to 
computing birational invariants of coverings.

\begin{ex}
\label{ex:4A2}
Let us consider the curve $\cC\subset S=\PP^2_{(1,2,1)}$ with equation
\[
F(x,y,z):=-4 x^{3} z^{3} - 27 z^{6} - 18 x y z^{3} + x^{2} y^{2} + 4 y^{3}=0.
\] 
This curve of degree~$6$ has 4 points of type $\mathbb{A}_2$ and does not pass
through the singular point~$P$ of the plane.
The only possible $k$ for which $H^1(Y,\cO_Y(L^{(k)}))$ may not vanish is $k=5$ where
the ideal is the maximal one.
We have 
\begin{equation*}
\pi^{(5)}: H^0(S,\mathcal{O}_{S}(5-4)) \longrightarrow\CC^4
\end{equation*}
and each point is in a different line passing through~$P$. Hence $\ker\pi^{(5)}=0$
and 
\[
\dim\coker\pi^{(5)}=\dim H^1(Y,\cO_Y(L^{(k)}))=2.
\]
\end{ex}

\begin{ex}
\label{ex:cusp23}
Let $\mathcal{C}$ be the curve of degree $d=6$ in $\PP^2_{w}$, $w=(3,2,1)$ given by $G = x^2 + y^3$.
Consider the cyclic branched covering $\rho: \tilde{X} \to \PP^2_w$ ramifying on $\mathcal{C}$.
This curve has only one singular point at $P = [0:0:1]$ and does not pass through the singular
points of~$\PP^2_{w}$. Resolving the singularity of the curve with a $(3,2)$-blow up gives rise
to one exceptional divisor $E$ with self-intersection number $-1/6$ and
$\pi^* \mathcal{C} = \hat{\mathcal{C}} + 6 E$, $K_{\pi} = 4E$.
Since $H^0(\PP^2_w, \mathcal{O}_{\PP^2_{w}}(k-6)) = 0$ for $k = 0,\dots,5$, the cokernel of $\pi^{(k)}$ is
$$
\coker \pi^{(k)} = \frac{\mathcal{O}_{\CC^2,0}}{\{ g \mid \mult_{E} \pi^{*} h > k - 5 \}},
$$
whose dimension is $1$ for $k=5$ and $0$ for $k=0,\ldots,4$. According to
Theorem~\ref{thm:conucleo_singular}, $\dim H^1(\tilde{X},\CC) = 2$ and the characteristic
polynomial of the monodromy of the covering acting on $H^1(\tilde{X},\CC)$ is
$(t-e^{\frac{2\pi i}{6}})(t-e^{\frac{2\pi i 5}{6}}) = t^2 - t + 1$.

Note that the $6$th-cyclic cover of $\PP^2$ ramified along any curve of degree $6$ having just a singular point
with the topological type of $x^2+y^3$ does not have any irregularity. Therefore the singular points of the 
ambient space might affect the irregularity of the covering, even though they have codimension~$2$.
\end{ex}

\subsection{A Zariski pair of curves on the weighted projective plane}\label{sec:zariski-pair}
\mbox{}

Consider the curve $\mathcal{C}_{\lambda} \subset S=\PP^2_{(1,1,3)}$ defined by the polynomial
$$
G_{\lambda}(x,y,z) = (\lambda y z + x z - x^{a} y^{b})^3 + (y z - x z + x^{a} y^{b})^3
+ (- \lambda y z + \lambda x z + x^{a} y^{b})^3
$$
with $\lambda\in \CC$ and $a+b=4$ so that it is quasi-homogeneous. As above, let $\rho_\lambda: \tilde{X}_\lambda \to S$
be the cyclic branched covering of degree $d=12$ ramifying on $\mathcal{C}_\lambda$.
Using $H^1(\tilde{X}_\lambda,\CC)$, which is
an invariant of the pair $(S,\mathcal{C}_\lambda)$, we will see that
$(S,\mathcal{C}_1)$ and $(S,\mathcal{C}_{\zeta})$ for $\zeta$ a $3$rd primitive root of unity, 
provide a Zariski pair for $(a,b)=(1,3)$.
However, the question remains open for $(a,b) = (2,2)$ and $\lambda = 1,\zeta$,
since $H^1(\tilde{X}_\lambda,\cO_{\tilde{X}_\lambda}) = 0$.

Assume from now on that $(a,b) = (1,3)$. The curve $\mathcal{C}_{\lambda}$ has only three singular points at
the origins of the projective plane, $P_1 = [1:0:0]$, $P_2=[0:1:0]$, $P_3=[0:0:1]$ and note that 
$(S,P_i)$ is regular at $i=1,2$ whereas $S_3=(S,P_3)=\frac{1}{3}(1,1)$.
From Theorem~\ref{thm:conucleo_singular} one needs to study the cokernel of the map
\begin{equation}\label{eq:pik-ex}
\pi^{(k)}: H^0(S,\mathcal{O}_{S}(k-5)) \longrightarrow
\frac{\cO_{\CC^2,0}}{\cM_{\cC_\lambda,P_1}^{(k)}}
\oplus
\frac{\cO_{\CC^2,0}}{\cM_{\cC_\lambda,P_2}^{(k)}}
\oplus
\frac{\cO_{S_3,\zeta^{(k-5)}}}{\mathcal{M}_{\cC_\lambda,P_3}^{(k)}}
\end{equation}
for $k=0,\ldots,11$. To understand the three vector spaces on the right-hand side
one resolves the singularities of the curve as shown in Figure~\ref{fig:res}. The local type of the 
surface singularities and the self-intersections are shown for convenience.

\begin{figure}[ht]
\begin{center}
\begin{tikzpicture}
\begin{scope}
\coordinate (A) at (0,2);
\coordinate (B) at (2,0);
\coordinate (C) at (-2,0.2);

\node[right=2pt] at (A) {$\frac{1}{3}(1,1)$};
\node[left=2pt] at (A) {$P_3$};
\fill (A) circle [radius=.1cm];

\node[above=3pt] at (C) {$P_1$};
\node[below=3pt] at (C) {$(\mathbb C^2,0)$};
\fill (C) circle [radius=.1cm];

\node[above=3pt] at (B) {$P_2$};
\node[below=3pt] at (B) {$(\mathbb C^2,0)$};
\fill (B) circle [radius=.1cm];

\draw[line width=1.2pt] (A) to[out=-90, in=150] (B) to[out=150,in=0] ($(B)+(C)-.25*(A)$) to[out=180,in=9] (C) to[out=0,in=-90] (A);

\node[below=.5cm] at ($.5*(C)+.5*(B)$) {$48$};
\end{scope}

\begin{scope}[xshift=8cm]
\coordinate (A) at (-1,4);
\coordinate (B) at (2,.5);
\coordinate (C) at (-1,-1);
\coordinate (D) at (1,-1);
\coordinate (E) at (1,2);
\coordinate (F) at (-2,.5);

\draw ($.5*(A)+.5*(E)+(0,-1)$) -- ($.5*(A)+.5*(E)+(0,1)$);
\fill ($.5*(A)+.5*(E)+(0,.75)$) circle [radius=.1cm];
\node[left] at ($.5*(A)+.5*(E)+(0,.75)$) {$\frac{1}{3}(1,2)$};
\fill ($.5*(A)+.5*(E)+(0,-.2)$) circle [radius=.1cm];
\node[left] at ($.5*(A)+.5*(E)+(0,-.2)$) {$\frac{1}{6}(1,1)$};
\node[right] at ($.5*(A)+.5*(E)+(-.2,1.2)$) {$-\frac{1}{6}$};
\node[below right] at ($.5*(A)+.5*(E)+(-.2,-1)$) {$E_3$};

\draw[rotate around={45:($.5*(D)+.5*(B)$)}] ($.5*(B)+.5*(D)+(0,-1)$) node[below=3] {$-\frac{1}{12}$} -- ($.5*(B)+.5*(D)+(0,1.5)$);
\fill[rotate around={45:($.5*(D)+.5*(B)$)}] ($.5*(B)+.5*(D)+(0,.75)$) circle [radius=.1cm];
\node[right] at ($.5*(D)+.5*(B)+(-.25,.5)$) {$\frac{1}{4}(-1,3)$};
\node[right] at ($.5*(D)+.5*(B)+(-1.5,.8)$) {$E_2$};

\fill[rotate around={45:($.5*(D)+.5*(B)$)}] ($.5*(B)+.5*(D)-(0,.75)$) circle [radius=.1cm];
\node[right=10pt] at ($.5*(D)+.5*(B)-(-.25,.5)$) {$\frac{1}{3}(-1,4)$};

\draw[rotate around={-45:($.5*(F)+.5*(C)$)}] ($.5*(F)+.5*(C)+(0,-1)$) node[below=3] {$-\frac{1}{30}$} -- ($.5*(F)+.5*(C)+(0,1.5)$);
\node[right] at ($.5*(F)+.5*(C)+(-1.5,.5)$) {$\frac{1}{3}(-1,10)$};
\fill[rotate around={-45:($.5*(F)+.5*(C)$)}] ($.5*(F)+.5*(C)+(0,.75)$) circle [radius=.1cm];
\node[right] at ($.5*(F)+.5*(C)+(.8,1.25)$) {$E_1$};

\node[right] at ($.5*(F)+.5*(C)-(2.5,.5)$) {$\frac{1}{10}(-1,3)$};
\fill[rotate around={-45:($.5*(F)+.5*(C)$)}] ($.5*(F)+.5*(C)-(0,.75)$) circle [radius=.1cm];

\draw[line width=1.2] (0,1) ellipse [x radius=.75cm,y radius=1.5cm] ;
\node at (0,-.8) {$0$};
\end{scope}

\end{tikzpicture}
\end{center}
\caption{Resolution of $\cC_\lambda$}
\label{fig:res}
\end{figure}

Locally at $P_3$ the polynomial $G_\lambda$ can be written as
$$
(x+\lambda y)^3 + 3xy^3 \left( y^2 - 2(\lambda^2+\lambda+1)xy + \lambda^2x^2 \right)
+ 3x^2y^6(y+\lambda x)+x^3y^9
$$
and thus changing the local coordinates with $(x,y) = (x_3,\lambda^2(y_3-x_3))$
which is compatible with the group action, one obtains 
\begin{equation}
\label{eq:deformation}
\left( \lambda^3-1 \right)
\left( A_{\lambda,3} x_3^3+A_{\lambda,4} x_3^2y_3+ A_{\lambda,5} x_3y_3^2 \right)
+ B_\lambda x_3^6 + \lambda^6 y_3^3 + \sum_{i,j} C_{\lambda ij} x_3^i y_3^j
\end{equation}
with $i+2j>6$, $A_{\lambda,3}=\lambda^2(3\lambda^2 + 2\lambda + 3)$,
$A_{\lambda,4}=-3\lambda^2(2\lambda^2 + \lambda + 1)$,
$A_{\lambda,5}=3\lambda^4$, and 
$B_\lambda=-3\lambda^8(3\lambda^2 + 2\lambda + 3)$. 
If $\lambda=1,\zeta$, then the blow-up with respect to the weight $(1,2)$ resolves the curve at the 
point~$P_3\in S_3=\frac{1}{3}(1,1)$. In addition, this shows that $(G_\lambda,P_3)$, $\lambda=1,\zeta$, 
have the same topological type, since they have the same resolution graph, as shown in Figure~\ref{fig:res}. 
Moreover, note from equation~\eqref{eq:deformation} that the deformations in this family near $\cC_1$ or 
$\cC_\zeta$ are not equisingular.
Let us denote by $E_3$ the exceptional divisor, then $\mult_{E_3}\pi^*(x_3^iy_3^j)=i+2j$ 
and hence $m_3 = 6$, $\nu_3 = 3$ (see section~\ref{sec:Qres}). 
Note that the equation $x_3^6 + y_3^3$ defines an \emph{irreducible} curve in $S_3$. 
Using these local coordinates $\mathcal{O}_{S_3} = \CC\{ x_3^3, x_3^2y_3,x_3y_3^2,y_3^3\}$.
The third vector space in~\eqref{eq:pik-ex} can be described as
\begin{equation}\label{eq:M_P3}
\frac{\cO_{S_3,\zeta^{(k-5)}}}{\cM_{\cC_\lambda,P_3}^{(k)}}
= \frac{\cO_{S_3,\zeta^{(k+1)}}}{ \left\{ g \ \big| \ \mult_{E_3} \pi^{*} g
> \dfrac{k}{2} - 3 \right\}} \cong 
\begin{cases}
0 & \text{ if }\ \ k = 0,\ldots,7, \\
\langle x_3^{\overline{k+1}} \rangle_\CC & \text{ if } \ \ k = 8,\ldots,11,
\end{cases}
\end{equation}
where $\overline{k+1}$ denotes $k+1$ modulo 3. This is a consequence of the fact that
$\mathcal{O}_{S_3,\zeta^0} = \mathcal{O}_{S_3}$,
$\mathcal{O}_{S_3,\zeta^1} = \langle x_3,y_3 \rangle\mathcal{O}_{S_3}$,
and $\mathcal{O}_{S_3,\zeta^2} = \langle x_3^2, x_3y_3, y_3^2 \rangle\mathcal{O}_{S_3}$.

For $P_2 = [0:1:0]$ one uses the change of coordinates on $G_\lambda(x,1,z)$,
$(x, z) = (x_2, z_2-x_2)$ to get a polynomial
of the form $\beta_\lambda x_2^4 + z_2^3 + \sum_{ij} b_{\lambda ij} x_2^i z_2^j$ with $3i+4j>12$ and $\beta_\lambda \neq 0$.
The blow-up with respect to $(3,4)$ resolves the singular point. The exceptional divisor is denoted by $E_2$
and $m_2 = 12$, $\nu_2 = 7$. The second summand in~\eqref{eq:pik-ex} is
\begin{equation}\label{eq:M_P2}
\frac{\mathcal{O}_{\CC^2}}{\mathcal{M}_{\mathcal{C}_{\lambda},P_2}^{(k)}}
= \frac{\mathcal{O}_{\CC^2}}{ \left\{ g \ \big| \ \mult_{E_3} \pi^{*} g > k - 7 \right\}} \cong 
\begin{cases}
0 & \text{ if }\ \ k = 0,\ldots,6, \\
\CC & \text{ if } \ \ k = 7,8,9, \\
\langle 1,x_2 \rangle_\CC & \text{ if } \ \ k = 10, \\
\langle 1,x_2,z_2 \rangle_\CC & \text{ if } \ \ k = 11.
\end{cases}
\end{equation}

Finally, for $P_1 = [1:0:0]$ one blows up $G_\lambda(1,y_1, \lambda^2(z_1-y_1^3)) =
\alpha_\lambda y_1^{10} + z_1^3 + \sum_{ij} a_{\lambda ij} y_1^i z_1^j$ with $3i+10j>30$ and $\alpha_\lambda \neq 0$
with weight vector $(3,10)$ to resolve the singular point. The exceptional divisor $E_1$ gives
rise to $m_1=30$ and $\nu_1=13$. The first vector space in~\eqref{eq:pik-ex} is
\begin{equation}\label{eq:M_P1}
\frac{\mathcal{O}_{\CC^2}}{ \left\{ g \ \big| \ \mult_{E_3} \pi^{*} g > \dfrac{5k}{2} - 13 \right\}}
\cong
\begin{cases}
0 & \text{ if } \ \ k = 0,\ldots,5, \\
\CC & \text{ if } \ \ k = 6, \\
\langle 1,y_1 \rangle_\CC & \text{ if } \ \ k = 7, \\
\langle 1,y_1,y_1^2 \rangle_\CC & \text{ if } \ \ k = 8, \\
\langle 1,y_1,y_1^2,y_1^3 \rangle_\CC & \text{ if } \ \ k = 9, \\
\langle 1,y_1,y_1^2,y_1^3,y_1^4,z_1 \rangle_\CC & \text{ if } \ \ k = 10, \\
\langle 1,y_1,y_1^2,y_1^3,y_1^4,z_1, y_1 z_1 \rangle_\CC & \text{ if } \ \ k = 11.
\end{cases}
\end{equation}
Using the previous calculations one can check that $H^1(Y,\cO_Y(L^{(k)})) = 0$ for any $k \neq 10$
and $\lambda = 1,\zeta$ by studying the cokernel of~$\pi^{(k)}$.

To finish we give the details for $k=10$. In order to study the cokernel of $\pi^{(10)}$, let us fix
a basis $\mathcal{B}_1 := \{ x^5, x^4y, x^3y^2, x^2y^3, xy^4, y^5, x^2 z, xyz, y^2z \}$ of
$H^0(\PP^2_w,\mathcal{O}_{\PP^2_w}(k-5))$ and a basis
$\mathcal{B}_2 = \langle 1,y_1,y_1^2,y_1^3,y_1^4,z_1 \rangle_\CC \oplus \langle 1,x_2 \rangle_\CC
\oplus \langle x_3^2 \rangle_\CC =: \mathcal{B}_{21} \oplus \mathcal{B}_{22} \oplus \mathcal{B}_{23}$
of the vector space on the right-hand side in~\eqref{eq:pik-ex}
computed above, see~\eqref{eq:M_P3}, \eqref{eq:M_P2}, \eqref{eq:M_P1}. The matrix associated with
$\pi^{(10)}$ in the bases $\mathcal{B}_1$ and $\mathcal{B}_2$ becomes
$$
A = \begin{pmatrix}
1 & 0 & 0 & 0 & 0 & 0 & 0 & 0 & 0 \\
0 & 1 & 0 & 0 & 0 & 0 & 0 & 0 & 0 \\
0 & 0 & 1 & 0 & 0 & 0 & 0 & 0 & 0 \\
0 & 0 & 0 & 1 & 0 & 0 & -\lambda^2 & 0 & 0 \\
0 & 0 & 0 & 0 & 1 & 0 & 0 & -\lambda^2 & 0 \\
0 & 0 & 0 & 0 & 0 & 0 & \lambda^2 & 0 & 0 \\
0 & 0 & 0 & 0 & 0 & 1 & 0 & 0 & 0 \\
0 & 0 & 0 & 0 & 1 & 0 & 0 & 0 & -1 \\
0 & 0 & 0 & 0 & 0 & 0 & 1 & -\lambda^2 & \lambda
\end{pmatrix}.
$$
To clarify how this matrix was built let us recall the definition of $\pi^{(10)}$.
Given $F(x,y,z)$ its image under $\pi^{(10)}$ is
$$
F(1, y_1, \lambda^2(z_1-y_1^3))_{\mathcal{B}_{21}}
\oplus F(x_2, 1, z_2-x_2)_{\mathcal{B}_{22}} \oplus F(x_3,\lambda^2(y_3-x_3),1)_{\mathcal{B}_{23}},
$$
where the subindex indicates coordinates in the corresponding basis. This way for example
$\pi^{(10)}(xyz) = (\lambda^2 y_1 z_1 - \lambda^2 y_1^4)_{\mathcal{B}_{21}}
\oplus (x_2 z_2 - x_2^2)_{\mathcal{B}_{22}}
\oplus (\lambda^2 x_3 y_3 - \lambda^2 x_3^2)_{\mathcal{B}_{23}}$ which produces
the vector $(0,0,0,0,-\lambda^2,0,0,0,-\lambda^2)$, namely the $8$th column of $A$.

For $\lambda = \zeta$ the rank of $A$ is $9$, thus $\coker \pi^{(10)} = 0$. However, for $\lambda = 1$
the rank is~$8$ and then the kernel and the cokernel has dimension $1$. The kernel is generated
by the quasi-homogeneous polynomial 
$F=xy^4 + xyz + y^2z=y(x y^3 +x z+ y z)$ and the cokernel is generated for instance by $x_3^2 \in \mathcal{B}_{23}$.
This shows that $\dim H^1(\tilde{X}_{1},\CC) = 2$ while $\dim H^1(\tilde{X}_{\zeta},\CC) = 0$.
Note that the expected dimension of the cokernel is 0 since both spaces have the same dimension.
The existence of the special curve $\{F=0\}$ whose equation is in the kernel of this map, satisfying certain local 
properties at $P_1$, $P_2$, and $P_3$ determines the non-trivial irregularity for $\tilde{X}_{1}$. This phenomenon, 
classically referred to as \emph{superabundance}, can be geometrically seen in Figure~\ref{fig:superabundance}. 

\begin{figure}[ht]
\begin{center}
\begin{tikzpicture}
\coordinate (A) at (0,2);
\coordinate (B) at (2,0);
\coordinate (C) at (-2,0.2);
\node[left=2pt] at (A) {$P_3$};
\fill (A) circle [radius=.1cm];
\node[below=3pt] at (C) {$P_1$};
\fill (C) circle [radius=.1cm];
\node[below=3pt] at (B) {$P_2$};
\fill (B) circle [radius=.1cm];
\draw[line width=1.2pt] (A) to[out=-90, in=150] (B) to[out=150,in=0] ($(B)+(C)-.25*(A)$) to[out=180,in=9] (C) to[out=0,in=-90] (A);
\draw[line width=1.2pt,color=blue,dashed] (C) to[out=0, in=200] ($(C)+(1.7,0.3)$) to[out=30,in=-90] (A) to[out=90,in=-30] (B) to[out=150,in=0] (C);
\draw[line width=1.2pt,color=blue,dashed] ($1.3*(C)-.3*(A)$) -- ($1.3*(A)-.3*(C)$);
\end{tikzpicture}
\end{center}
\caption{Superabundance in $\cC_1$}
\label{fig:superabundance}
\end{figure}


\begin{thebibliography}{10}

\bibitem{Artal94}
E.~Artal, \emph{Sur les couples de {Z}ariski}, J. Algebraic Geom. \textbf{3}
  (1994), no.~2, 223--247.

\bibitem{ACM-Nemethi60}
E.~Artal, J.I. Cogolludo-Agust\'{\i}n, and J.~Mart\'{\i}n-Morales, 
\emph{Cremona transformations of weighted projective planes, {Z}ariski
  pairs, and rational cuspidal curves}, Preprint available at 
  \texttt{arXiv:2001.07232 [math.AG]}, 2020.
  
\bibitem{AMO-Intersection}
E.~Artal, J.~Mart\'{\i}n-Morales, and J.~Ortigas-Galindo, \emph{Intersection
  theory on abelian-quotient {$V$}-surfaces and {$\bf Q$}-resolutions}, J.
  Singul. \textbf{8} (2014), 11--30.

\bibitem{Blache95}
R.~Blache, \emph{Riemann-{R}och theorem for normal surfaces and applications},
  Abh. Math. Sem. Univ. Hamburg \textbf{65} (1995), 307--340.

\bibitem{Blickle-Lazarsfeld-informal}
M.~Blickle and R.~Lazarsfeld, \emph{An informal introduction to multiplier
  ideals}, Trends in commutative algebra, Math. Sci. Res. Inst. Publ., vol.~51,
  Cambridge Univ. Press, Cambridge, 2004, pp.~87--114.

\bibitem{jiJM-correction}
J.I. Cogolludo-Agust{\'i}n and J.~Mart\'{\i}n-Morales, \emph{The correction
  term for the {R}iemann--{R}och formula of cyclic quotient singularities and
  associated invariants}, Rev. Mat. Complut. \textbf{32} (2019), no.~2,
  419--450.

\bibitem{Dolgachev82}
I.~Dolgachev, \emph{Weighted projective varieties}, Group actions and vector
  fields ({V}ancouver, {B}.{C}., 1981), Lecture Notes in Math., vol. 956,
  Springer, Berlin, 1982, pp.~34--71.

\bibitem{es:82}
H.~Esnault, \emph{Fibre de {M}ilnor d'un c\^one sur une courbe plane
  singuli\`ere}, Invent. Math. \textbf{68} (1982), no.~3, 477--496.

\bibitem{Esnault-Viehweg82}
H.~Esnault and E.~Viehweg, \emph{Rev\^etements cycliques}, Algebraic threefolds
  ({V}arenna, 1981), Lecture Notes in Math., vol. 947, Springer, Berlin-New
  York, 1982, pp.~241--250.

\bibitem{Fulton-Intersection}
W.~Fulton, \emph{Intersection theory}, second ed., Ergebnisse der Mathematik
  und ihrer Grenzgebiete. 3. Folge. A Series of Modern Surveys in Mathematics,
  vol.~2, Springer-Verlag, Berlin, 1998.

\bibitem{Libgober-alexander}
A.~Libgober, \emph{Alexander polynomial of plane algebraic curves and cyclic
  multiple planes}, Duke Math. J. \textbf{49} (1982), no.~4, 833--851.

\bibitem{Libgober-characteristic}
\bysame, \emph{Characteristic varieties of algebraic curves}, Applications of
  algebraic geometry to coding theory, physics and computation (Eilat, 2001),
  Kluwer Acad. Publ., Dordrecht, 2001, pp.~215--254.

\bibitem{Loeser-Vaquie-Alexander}
F.~Loeser and M.~Vaqui{\'e}, \emph{Le polyn\^ome d'{A}lexander d'une courbe
  plane projective}, Topology \textbf{29} (1990), no.~2, 163--173.

\bibitem{Mumford-topology}
D.~Mumford, \emph{The topology of normal singularities of an algebraic surface
  and a criterion for simplicity}, Inst. Hautes \'{E}tudes Sci. Publ. Math.
  (1961), no.~9, 5--22.

\bibitem{Nemethi-Poincare}
A.~N\'emethi, \emph{Poincar\'e series associated with surface singularities},
  Singularities {I}, Contemp. Math., vol. 474, Amer. Math. Soc., Providence,
  RI, 2008, pp.~271--297.

\bibitem{Nori-zariski}
M.V. Nori, \emph{{Z}ariski's conjecture and related problems}, Ann. Sci.
  \'Ecole Norm. Sup. (4) \textbf{16} (1983), no.~2, 305--344.

\bibitem{Sabbah-Alexander}
C.~Sabbah, \emph{Modules d'{A}lexander et {$\mathcal D$}-modules}, Duke Math.
  J. \textbf{60} (1990), no.~3, 729--814.

\bibitem{Sakai84}
F.~Sakai, \emph{Weil divisors on normal surfaces}, Duke Math. J. \textbf{51}
  (1984), no.~4, 877--887.

\bibitem{Steenbrink77}
J.H.M. Steenbrink, \emph{Mixed {H}odge structure on the vanishing cohomology},
  Real and complex singularities ({P}roc. {N}inth {N}ordic {S}ummer
  {S}chool/{NAVF} {S}ympos. {M}ath., {O}slo, 1976), Sijthoff and Noordhoff,
  Alphen aan den Rijn, 1977, pp.~525--563.

\bibitem{Zariski-problem}
O.~Zariski, \emph{On the {P}roblem of {E}xistence of {A}lgebraic {F}unctions of
  {T}wo {V}ariables {P}ossessing a {G}iven {B}ranch {C}urve}, Amer. J. Math.
  \textbf{51} (1929), no.~2, 305--328.

\bibitem{Zariski-irregularity}
\bysame, \emph{On the irregularity of cyclic multiple planes}, Ann. of Math.
  (2) \textbf{32} (1931), no.~3, 485--511.

\end{thebibliography}

\providecommand{\bysame}{\leavevmode\hbox to3em{\hrulefill}\thinspace}
\providecommand{\MR}{\relax\ifhmode\unskip\space\fi MR }
\providecommand{\MRhref}[2]{%
  \href{http://www.ams.org/mathscinet-getitem?mr=#1}{#2}
}
\providecommand{\href}[2]{#2}

\end{document}